\documentclass[12pt]{article}
\usepackage[utf8]{inputenc}
\usepackage[T1]{fontenc}
\usepackage{lmodern}
\usepackage{microtype}
\usepackage[top=1.6cm, bottom=2.0cm, left=3cm, right=3cm]{geometry}
\usepackage{amsmath,amssymb,amsthm,mathtools,tikz-cd,xfrac,stmaryrd}

\usepackage[pdftitle={The odd spectral localiser via asymptotic morphisms and 
quasi-projections},
pdfauthor={Yuezhao Li},
pdfsubject={Mathematics}
]{hyperref}
\usepackage{cleveref} 
\Crefname{appendix}{Appendix}{Appendices} 

\usepackage{tikz}
\usetikzlibrary{cd} 

\usepackage{enumitem}
\setlist[enumerate,1]{label=\textup{(\arabic*)},itemsep=0pt}
\setlist[enumerate,2]{label=\textup{(\alph*)},itemsep=0pt}
\setlist[itemize,1]{itemsep=0pt}
\setlist[itemize,2]{label=$\circ$,itemsep=0pt}

\usepackage[lite,alphabetic]{amsrefs}

\renewcommand*{\PrintDOI}[1]{\href{http://dx.doi.org/\detokenize{#1}}{doi:
\detokenize{#1}}}

\BibSpec{book}{%
  +{}  {\PrintPrimary}                {transition}
  +{,} { \textit}                     {title}
  +{.} { }                            {part}
  +{:} { \textit}                     {subtitle}
  +{,} { \PrintEdition}               {edition}
  +{}  { \PrintEditorsB}              {editor}
  +{,} { \PrintTranslatorsC}          {translator}
  +{,} { \PrintContributions}         {contribution}
  +{,} { }                            {series}
  +{,} { \voltext}                    {volume}
  +{,} { }                            {publisher}
  +{,} { }                            {organization}
  +{,} { }                            {address}
  +{,} { \PrintDateB}                 {date}
  +{,} { }                            {status}
  +{}  { \parenthesize}               {language}
  +{}  { \PrintTranslation}           {translation}
  +{;} { \PrintReprint}               {reprint}
  +{.} { }                            {note}
  +{.} {}                             {transition}
  +{}  { \PrintDOI}                   {doi}
  +{}  { available at \url}           {eprint}
  +{}  {\SentenceSpace \PrintReviews} {review}
}

\numberwithin{figure}{section}
\numberwithin{equation}{section}
\theoremstyle{plain}
\newtheorem{theorem}[equation]{Theorem}
\newtheorem*{theorem*}{Theorem} 
\newtheorem{lemma}[equation]{Lemma}
\newtheorem{proposition}[equation]{Proposition}

\newtheorem{corollary}[equation]{Corollary}
\theoremstyle{definition}
\newtheorem{definition}[equation]{Definition}
\theoremstyle{remark}
\newtheorem{remark}[equation]{Remark}
\newtheorem{example}[equation]{Example}

\newcommand{\N}{\mathbb{N}}
\newcommand{\Z}{\mathbb{Z}}

\newcommand{\R}{\mathbb{R}}
\newcommand{\C}{\mathbb{C}}

\newcommand{\Mat}{\mathbb{M}}
\newcommand{\Cpt}{\mathbb{K}}
\newcommand{\Bdd}{\mathbb{B}}

\newcommand{\K}{\relax\ifmmode\operatorname{K}\else\textup{K}\fi}
\newcommand{\KK}{\relax\ifmmode\operatorname{KK}\else\textup{KK}\fi}
\newcommand{\E}{\relax\ifmmode\operatorname{E}\else\textup{E}\fi}
\newcommand{\st}{\relax\ifmmode{}^*\else{}\textup{*}\fi}
\newcommand{\Cst}{\relax\ifmmode\mathrm{C}^*\else\textup{C*}\fi}
\newcommand{\Cont}{\mathrm{C}}
\newcommand{\Co}{\mathrm{C}_0}
\newcommand{\Cc}{\mathrm{C}_\mathrm{c}}

\newcommand{\Fin}{\operatorname{Fin}}
\newcommand{\End}{\operatorname{End}} 
\newcommand{\Ped}{\operatorname{Ped}} 
\newcommand{\dom}{\operatorname{dom}} 
\newcommand{\spc}{\operatorname{spec}} 
\newcommand{\Hom}{\operatorname{Hom}} 
\newcommand{\id}{\operatorname{id}} 
\newcommand{\ev}{\operatorname{ev}} 
\newcommand{\tr}{\operatorname{tr}} 
\newcommand{\sig}{\operatorname{sig}} 
\newcommand{\rank}{\operatorname{rank}} 
\newcommand{\Ad}{\operatorname{Ad}} 

\renewcommand{\Re}{\operatorname{Re}} 
\newcommand{\defeq}{\mathrel{\vcentcolon=}}

\let\amsamp=& 
\DeclarePairedDelimiter{\set}{\lbrace}{\rbrace} 
\DeclarePairedDelimiter{\abs}{\lvert}{\rvert} 
\DeclarePairedDelimiter{\norm}{\lVert}{\rVert} 
\DeclarePairedDelimiter{\ket}{\lvert}{\rangle} 
\DeclarePairedDelimiter{\bra}{\langle}{\rvert} 
\DeclarePairedDelimiter{\braket}{\langle}{\rangle} 
\DeclarePairedDelimiterX{\braketvert}[3]{\langle}{\rangle}{#1\,\delimsize\vert\,\mathopen{}#2\mathopen{}\,\delimsize\vert\,#3} 
\makeatletter
\renewcommand{\@fnsymbol}[1]{}
\makeatother

\usepackage{authblk}
\title{\vspace{-0.5em}The odd spectral~localiser via asymptotic~morphisms
and quasi-projections}
\author{Yuezhao Li}
\author{Bram Mesland\thanks{Email:
\texttt{\href{mailto:y.li@math.leidenuniv.nl}{y.li@math.leidenuniv.nl}},
\texttt{\href{mailto:b.mesland@math.leidenuniv.nl}{b.mesland@math.leidenuniv.nl}}.
}}
\affil{Mathematical Institute, Leiden University\thanks{Address:
Niels Bohrweg 1, 2333 CA Leiden, Netherlands
}}
\date{\today}

\begin{document}
\maketitle
\setcounter{section}{0}
\renewcommand\theHsection{P1.\thesection}
\renewcommand\thesection{\arabic{section}}
\vspace{-0.5cm}
\begin{abstract}
We describe the index pairing between an odd K-theory class and an odd
unbounded Kasparov module by a pair of quasi-projections, supported on a
submodule obtained from a finite spectral truncation.
We achieve this by pairing the K-theory class with an
asymptotic morphism determined by the unbounded Kasparov module.
We interpret the spectral localiser of Loring and
Schulz-Baldes as an instance of such an index pairing.
\end{abstract}

\tableofcontents

\section{Introduction}
In this note, we provide a quasi-projection representative of the index
pairing between a unitary~\( v\in A \) and an odd unbounded Kasparov~\( A
\)-\( B \)--module~\( (\mathcal{A},E,D) \), where~\( A \) and~\( B \) are
\Cst-algebras and~\( A \) is unital. These quasi-projections are supported
on a complemented submodule of~\( E\oplus E \) obtained from a finite
spectral truncation onto the low-lying spectrum of~\( D\oplus D \). In
particular, we realise the odd spectral localiser of Loring and
Schulz-Baldes
\cites{Loring-SBaldes:Odd_spectral_localiser,Loring-SBaldes:Odd_spectral_localiser_sf}
as such an index pairing for the case~\( B=\C \), and the odd semi-finite
spectral localiser of Schulz-Baldes and Stoiber
\cite{SBaldes-Stoiber:Semifinite_spectral_localizer} as a~\( \tau
\)-numerical index pairing where~\( B \) is unital and~\( \tau\colon B\to\C
\) is a finite, faithful trace.   Our work provides a new proof of their
results, and allows for interpreting the odd spectral localiser as an
E-theoretic index pairing.

\subsubsection*{The odd index pairing}
An index pairing is, by definition, the pairing between a K-theory class
and a KK-theory class, via the celebrated Kasparov product
\begin{equation}\label{eq:KK_index_pairing}\tag{KK}
    \underbrace{\KK_i(\C,A)}_{\simeq\K_i(A)}
    \times\KK_j(A,B)\to\underbrace{\KK_{i+j}(\C,B)}_{\simeq\K_{i+j}(B)}.
\end{equation}  
We call it the odd--odd index pairing if both~\( i \) and~\( j \) are odd.
If~\( B=\C \), then the odd--odd index pairing can be realised as a
Fredholm index as follows. Let~\( A \) be a unital
\Cst-algebra,~\( (\mathcal{A},\mathcal{H},D) \) be an odd
spectral triple over~\( A \) and~\( v\in A \) be a unitary. Then~\( v \)
defines an odd K-theory class~\( [v]\in\K_1(A) \) and~\(
(\mathcal{A},E,D) \) defines an odd \K-homology class~\(
[D]\in\K^1(A) \). Their Kasparov product~\( [v]\times[D] \) is represented
by the Fredholm index of the following operator:
\[ 
    PvP + (1-P) \in\Bdd(\mathcal{H})
\]
where~\( P\defeq\chi_{>0}(D)  \) is the positive spectral projection of~\(
D\). Since the Fredholm index is invariant under a compact perturbation,
one can replace~\( P \) by any operator of the form
\begin{equation}\label{eq:P_1}
    P_1\defeq \frac{\chi(D)+1}{2}
\end{equation}
where~\( \chi\colon \R\to[-1,1] \) is a chopping function (i.e.~\(
\lim_{x\to\pm\infty}\chi(x)=\pm 1 \)).

In order to compute the index of the Fredholm operator~\( PvP+1-P \), it is
necessary to understand the whole spectrum of~\( P \) (or~\( D \)) and~\( v
\), which are operators on the \emph{infinite-dimensional} Hilbert space~\(
\mathcal{H} \). Numerical computation with such objects usually requires an
approximation of~\( \mathcal{H} \) by its finite-dimensional subspaces, and
operators thereon by finite matrices.  However, every Fredholm operator on
a finite-dimensional Hilbert space necessarily has zero index. It is,
therefore, not even clear what finite-dimensional approximation really
means in this context.

\subsubsection*{The odd spectral localiser}
A breakthrough in the numerical computation of index pairings
has been made by Loring and Schulz-Baldes in a series of
articles
\cites{Loring-SBaldes:Odd_spectral_localiser,Loring-SBaldes:Even_spectral_localiser,Loring-SBaldes:Odd_spectral_localiser_sf,Viesca-Schober-SBaldes:Even_spectral_localiser_sf}.
Given an odd unital
spectral triple~\( (\mathcal{A},\mathcal{H},D) \) and a unitary (or
invertible)~\( v\in\mathcal{A} \), Loring and Schulz-Baldes
define the following unbounded, self-adjoint operator
\begin{equation}\label{eq:spectral_localiser}
    L_{\kappa}\defeq \begin{pmatrix}
    \kappa D & v \\ v^* & -\kappa D
    \end{pmatrix}
\end{equation}
and~\( L_{\kappa,\lambda} \) as its truncation onto
the subspace
\begin{equation}\label{eq:spectral_subspace_intro}
    \mathcal{H}_\lambda\defeq \chi_{\lambda}(D\oplus
    D)(\mathcal{H}\oplus \mathcal{H}),
\end{equation}
where~\( \chi_\lambda \) is the characteristic function of the interval~\(
[-\lambda,\lambda] \). The Hilbert space~\( \mathcal{H}_\lambda \)
has finite dimension because~\( D \) has discrete spectrum, thus~\(
L_{\kappa,\lambda} \) is a \emph{finite} matrix. The main theorem of
\cites{Loring-SBaldes:Odd_spectral_localiser,Loring-SBaldes:Odd_spectral_localiser_sf} says that
\begin{quote}
For sufficiently \emph{small} tuning parameter~\( \kappa>0 \) and
sufficiently \emph{large} spectral threshold~\( \lambda>0 \), the odd--odd
index pairing coincides with the half-signature of the
finite-dimensional matrix~\( L_{\kappa,\lambda} \). 
\end{quote}
A similar construction for the even--even index pairing was provided in
\cites{Loring-SBaldes:Even_spectral_localiser,Viesca-Schober-SBaldes:Even_spectral_localiser_sf}.

So far, two different proofs have appeared in the literature. 
One based on fuzzy spheres as well as a slightly different picture
of K-theory
(cf.~\cites{Loring-SBaldes:Odd_spectral_localiser,Loring-SBaldes:Even_spectral_localiser});
another based on spectral flow
(cf.~\cites{Loring-SBaldes:Odd_spectral_localiser_sf,Viesca-Schober-SBaldes:Even_spectral_localiser_sf}).

In this note, we provide a new proof, which clarifies that the spectral
localiser itself can be interpreted as a class in K-theory.
To achieve this, we work with a version of bivariant K-theory,
the \E-theory of Connes and Higson, to compute the index pairing. 
The resulting K-theory class in~\( \K_0(\Cpt_B(E)) \)  
is represented by the formal difference~\( [e_t]-[f_t] \) of a pair of 
\emph{quasi-projections} in~\( \Mat_2(\Cpt_B(E)^+) \), 
which is the image of~\( [v]\in\K_1(A) \) under an 
\emph{asymptotic morphism}~\( (\Phi_t^D)_{t\in[1,\infty)} \) 
constructed from the unbounded Kasparov module~\( (\mathcal{A},E,D) \).

Quasi-projections, or quasi-idempotents, represent K-theory classes in the
spirit of quantitative K-theory
(cf.~\cites{Oyono-Yu:Quantitative_K-theory,Chung:Quantitative_K-theory_Banach_algebras}).
The quasi-projection picture allows for explicit homotopies between the
quasi-projection representative~\( [e_t]-[f_t] \) and 
its truncation~\( [p^e_{t,\lambda}]-[p^f_{t,\lambda}] \)
onto a suitable submodule. 

The existence of such a
submodule is an extra assumption, but is always guaranteed if~\( B=\C \).
An odd Kasparov~\( A \)--\( \C \)-module is the same thing as an odd 
spectral triple~\( (\mathcal{A},\mathcal{H},D) \), 
and a spectral decomposition of~\(
\mathcal{H}\oplus\mathcal{H} \) can always be achieved from a spectral
projection of the unbounded, self-adjoint operator~\( D\oplus D \).
In this situation, we prove that the spectral localiser
\eqref{eq:spectral_localiser} is,
up to congruence by a positive, invertible matrix, given
by the formal difference of quasi-projections~\(
[p^e_{t,\lambda}]-[p^f_{t,\lambda}] \), thus 
represents the index pairing~\( \braket{[v],[D]}\). Hence,
we interpret the odd spectral localiser as a
pairing between an odd K-theory class (represented by a unitary) 
with an asymptotic morphism.

Our method also applies to a class of \emph{numerical} index pairings,
which are compositions of maps of the form
\begin{equation}\label{eq:numerical_index_pairing_intro}
    \K_1(A)\xrightarrow{\times[D]}\K_0(B)\xrightarrow{\tau_*}\R,
\end{equation}
where~\( [D]\in\KK_1(A,B) \) is the class represented by an unbounded
Kasparov~\( A \)--\( B \)-module~\( (\mathcal{A},E,D) \), and~\( \tau \) is
a finite, faithful trace on the unital \Cst-algebra~\( B \). 
We show that such a numerical index pairing can be represented by
a formal difference of~\( \tau \)-finite quasi-projections, and reproduce 
the \emph{semi-finite spectral
localiser} introduced by Schulz-Baldes and Stoiber
\cite{SBaldes-Stoiber:Semifinite_spectral_localizer} from such
quasi-projection representatives.

\subsubsection*{Outline}
We summarise here the basic strategy that allows us to represent the odd
index pairing~\( \braket{[v],[D]} \) by quasi-projections that are
supported on spectral submodules, and how the spectral localiser emerges
from this construction.

\begin{enumerate}
\item An asymptotic morphism is a collection of maps that
``asymptotically'' becomes a \st-homomorphism. From 
an odd, unbounded Kasparov~\( A \)-\( B \)--module 
\( (\mathcal{A},E,D) \), we construct an asymptotic morphism~\(
(\Phi_t^D)_{t\in [1,\infty)} \), where
(cf.~\Cref{app:asymptotic_morphism}):
\begin{equation}\label{eq:Phi_intro}
    \Phi^D_t\colon \Co(\R)\otimes A\to
    \Cpt_B(E),\quad \Phi_t^D(f\otimes a)\defeq
    f(t^{-1}D)a.
\end{equation}
It induces a map
\begin{equation}\label{eq:Phi_induced_map_intro}
    (\Phi_t^D)_*\colon\K_0(\mathrm{S}A)\to\K_0(B),
\end{equation}
which coincides with
the map~\( \K_1(A)\xrightarrow{\times[D]}\K_0(B) \) induced by the
Kasparov product with the class~\( [D]\in\KK_1(A,B) \), up to identifying
\( \K_1(A) \) with~\( \K_0(\mathrm{S}A) \) using the suspension
isomorphism. 

\item The map \eqref{eq:Phi_induced_map_intro} can be described as follows.
Let~\( [e]-[f]\in\K_0(\mathrm{S}A) \) be the image of~\(
[v]\in\K_1(A) \) under the suspension isomorphism given as in
\Cref{lem:odd_K-theory}. Set~\(
\check{e}_t\defeq\Phi_t(e) \) and~\( f_t\defeq \Phi_t(f) \). 
It turns out that
\( (f_t)_{t\in [1,\infty)} \) is a constant family of projections in this
case; and for sufficiently large~\( t \),~\( \check{e}_t \) is a
quasi-idempotent in~\( \Mat_2(\Cpt_B(E)^+) \). Then
the map \eqref{eq:Phi_intro} sends~\( [e]-[f]\in\K_0(\mathrm{S}A) \)
to~\( [\check{e}_t]-[f_t]\in\K_0(\Cpt_B(E)) \), which we may
further replace up to homotopy by~\( [e_t]-[f_t]\in\K_0(\Cpt_B(E)) \) for
\( e_t \) a quasi-projection. 

\item The ``truncation'' step involves an extra parameter~\( \lambda>0 \),
encoding the size of the truncation spectrum. Let~\( \widehat{E}\defeq
E\oplus E \) and assume that there is a direct sum decomposition~\(
\widehat{E}=E_{\lambda}^{\downarrow}\oplus E_{\lambda}^\uparrow \) into
\emph{complemented} submodules, which is ``spectral'' in the sense that 
\( E_{\lambda}^{\downarrow} \) corresponds to the low-lying spectrum of~\(
D\oplus D \). Then for sufficiently large~\(
\lambda \), we may reduce~\( e_t \) 
to its ``spectral diagonal'' up to homotopy of quasi-projections, 
namely, the direct sum of 
truncated operators~\( p^e_{t,\lambda} \) and~\( q^e_{t,\lambda} \) 
on two spectral submodules~\( E_\lambda^\downarrow \) and
\( E_{\lambda}^\uparrow \); and~\( f_t \) is already diagonal
for this decomposition. This yields an
equality of K-theory classes in~\( \K_0(\Cpt_B(E)^+) \):
\[ 
    [e_t]=[p^e_{t,\lambda}]+[q^e_{t,\lambda}],\quad
    [f_t]=[p^f_{t,\lambda}]+[q^f_{t,\lambda}].
\]
We also show that~\( [q^e_{t,\lambda}]=[q^f_{t,\lambda}] \) for
sufficiently large~\( \lambda \). This implies that the odd index pairing 
\( [e_t]-[f_t] \) can equivalently be represented by
\( [p^e_{t,\lambda}]-[p^f_{t,\lambda}] \). Both of them are now supported
on the spectral submodule~\( E_{\lambda}^{\downarrow} \).

\item A spectral triple~\( (\mathcal{A},\mathcal{H},D) \) is the same as an
unbounded Kasparov~\( A \)-\( \C \)--module, and a spectral 
decomposition of~\( \widehat{\mathcal{H}}=\mathcal{H}\oplus \mathcal{H} \) 
may be constructed from spectral projections of~\( D\oplus D \) 
as in \eqref{eq:spectral_subspace_intro}. 
If~\( (\mathcal{A},\mathcal{H},D) \) 
is unital, then~\( D \) has discrete spectrum, 
hence the spectral subspace~\(
\mathcal{H}_\lambda \) is finite-dimensional. We show that
the spectral localiser~\( L_{t^{-1},\lambda} \)
is congruent to~\( \mathcal{L}^e_{t^{-1},\lambda}\defeq 2p^e_{t,\lambda}-1 \). 
It therefore follows
that the half-signature of the spectral localiser coincides with the odd
index pairing. This yields a new proof of the main result in
\cites{Loring-SBaldes:Odd_spectral_localiser,Loring-SBaldes:Odd_spectral_localiser_sf}.

\item Similar ideas can be applied to compute numerical index pairings as
in \eqref{eq:numerical_index_pairing_intro}. In this case,
to accomodate the truncation step, one must pass to a larger
\Cst-algebra generated by all~\( \widehat{\tau} \)-finite projections, 
where~\( \widehat{\tau} \) is the transferred trace on~\( \Cpt_B(E) \).
Then we recover the main result of
\cite{SBaldes-Stoiber:Semifinite_spectral_localizer} in this special
setting.
\end{enumerate}

\subsubsection*{Related works}
Finite-dimensional representatives of K-homology classes
have been considered by Willett in
\cites{Willett:Approximate_ideal_structures,Willett:Bott_periodicity}. 
It was proven 
in~\cite{Willett:Approximate_ideal_structures}*{Section 7} that there exists
an asymptotic 
morphism~\( (\alpha^A_t\colon \mathrm{S}^2A\to A\otimes\Cpt)_{t\in [1,\infty)}\)
such that for each~\( t\in [1,\infty) \) and~\( a\in \mathrm{S}^2A \), 
then~\( \alpha^A_t(a) \) belongs to 
a matrix algebra~\( \Mat_n(A) \). 
In~\cite{Willett:Bott_periodicity}, a finite-dimensional representative of
the Dirac (inverse Bott) element is constructed based on the 
\emph{localisation algebra} of 
Yu~\cite{Yu:Localization_algebras_coarse_Baum-Connes}. 
The explicit computation in~\cite{Willett:Bott_periodicity} has been the
inspiration for computing the odd index pairing via asymptotic morphisms
and quasi-projections.

The localisation algebra provides a model for Kasparov theory
\cite{Dadarlat-Willett-Wu:Localization_algebras}, with the advantage that
the natural transformation~\( \KK_1\to\E_1 \) can be described in an
explicit way, cf.~\cite{Dadarlat-Willett-Wu:Localization_algebras}*{Section
5}. We expect that there is a close relation between this model and the
constructions in the present article.

Connes and van Suijlekom have studied spectral truncation in
\cite{Connes-vSuijlekom:Spectral_truncations}, in order to provide a
framework of coarse graining approximation of geometric space at a finite
resolution. This yields a spectral triple over an \emph{operator system}
(as opposed to a \Cst-algebra).  More recently, van Suijlekom has
introduced a version of K-theory for operator systems in
\cites{vSuijlekom:K-theory_operator_syetems_1,vSuijlekom:K-theory_operator_systems_2},
and shown that the spectral localiser yields a well-defined index map for
it. It is possible that results in the present article might shed some
light on how K-\emph{homology}, or more generally, \emph{bivariant}
K-theory, should be defined for operator systems.

\subsubsection*{Notation and convention}
We fix the following notation. Let~\( A \) be a \Cst-algebra. We
write~\( A^+ \) for its (minimal) unitisation, and~\( A_+ \) for its
positive cone. We write~\( \mathrm{S}A \) for the non-unital 
\Cst-algebra~\( \Co(\R)\otimes A \). A Hilbert~\( A \)-module is always
assumed to be a right~\( A \)-module and its inner product~\(
\braket{\cdot,\cdot} \) is~\( A \)-linear in the second entry.

\subsection*{Acknowledgement}
The work is part of NWO project 613.009.142
``\href{https://www.nwo.nl/projecten/613009142}{Noncommutative index theory
of discrete dynamical systems}''. YL thanks Guo Chuan Thiang and Xingni
Jiang for helpful discussions, and thanks Hang Wang for 
helpful discussion and
warm hospitality during his visit at East China Normal University.
BM thanks Koen van den Dungen, Jens Kaad and Adam Rennie for insightful
conversations. Both BM and YL thank Nigel Higson, Ralf Meyer and Rufus
Willett for illuminating comments.

\section{Preliminaries}

\subsection{Quasi-projections and quasi-idempotents}
Let~\( A \) be a \Cst-algebra and let~\( 0<\varepsilon<\frac{1}{4} \). An
\( \varepsilon \)-\emph{quasi-idempotent} is an element~\( e\in A \)
satisfying~\( \norm{e^2-e}<\frac{1}{4} \). If~\( \varepsilon=\frac{1}{4}
\), then we also call~\( e \) a \emph{quasi-idempotent} for short. An
\( \varepsilon \)-\emph{quasi-projection} is a self-adjoint~\( \varepsilon
\)-quasi-idempotent.

Let~\( e_0,e_1\in A \) be quasi-idempotents. 
A \emph{homotopy} between~\( e_0 \)
and~\( e_1 \) is a norm-continuous path 
\( e\colon [0,1]\to A \), such that each
\( e(t) \) is a quasi-idempotent, and that~\( e(0)=e_0 \) and~\( e(1)=e_1
\). Similarly, a homotopy between quasi-projections is a norm-continuous
path of quasi-projections connecting them.

If~\( e\in A \) 
is a quasi-idempotent, then~\( \spc(e) \) does not meet the line
\[
   L\defeq\set*{z\in\C\;\middle|\;\Re z=\sfrac{1}{2}}
\]
in the complex plane. Hence the function
\[ 
    \kappa_0\colon \C\setminus L\to\{0,1\},\quad \kappa_0(z)\defeq\begin{cases}
    1, & \Re z>\sfrac{1}{2};\\
    0, & \Re z<\sfrac{1}{2},
    \end{cases}
\]
is holomorphic on~\( \spc(e) \). Holomorphic functional calculus then gives
a projection~\( \kappa_0(e)\in A \) as well as a
well-defined class~\( [e]\defeq [\kappa_0(e)]\in\K_0(A) \). Similarly,
every quasi-idempotent~\( e\in\Mat_n(A) \) determines a class in~\( \K_0(A)
\). This allows for ``defining''~\( \K_0(A) \) using quasi-idempotents or
quasi-projections in the same way as projections, in the following sense:

\begin{proposition}[cf.~\cite{Willett-Yu:Higher_index_theory}*{Proposition
2.2.9}]\label{prop:K-theory_quasi-idempotents}
Let~\( A \) be a unital \Cst-algebra. 
Then~\( \K_0(A) \) is isomorphic to the Grothendieck group of the monoid
\( V(A) \), where~\( V(A) \) consists of homotopy-equivalence classes of
quasi-idempotents in~\( \Mat_{\infty}(A) \), and the sum is given by~\(
[e]+[f]\defeq\left[\left(\begin{smallmatrix}
    e & 0 \\ 0 & f
\end{smallmatrix}\right)\right] \). 
\end{proposition}

As opposed to projections or idempotents, quasi-idempotents and
quasi-projections form open sets in~\( A \). This is convenient for
constructing homotopies between them, and allowing more representatives for
classes in~\( \K_0(A) \). We provide below
a few quantitative estimates related to quasi-idempotents. Many
similar but not identical results have been used in the context of
quantitative K-theory for \Cst-algebras or Banach algebras, 
cf.~\cites{Oyono-Yu:Quantitative_K-theory,Chung:Quantitative_K-theory_Banach_algebras}.

\begin{lemma}\label{lem:estimate_f2-f}
Let~\( e,f\in A \), then we have the following inequality\textup{:}
\[ 
    \norm{f^2-f}\leq(2\norm*{e}+\norm{e-f}+1)\norm*{e-f}+\norm{e^2-e}.
\]
\end{lemma}
\begin{proof}
This follows from
\[ 
    f^2-f=(f-e)(f-e)+e(f-e)+(f-e)e+(e-f)+(e^2-e). \qedhere
\]
\end{proof}

\begin{lemma}\label{lem:homotopy_quasi-idempotents}
Let~\( e\in A \) be an~\( \varepsilon_e \)-quasi-idempotent,~\( f\in A \) an
\( \varepsilon_f \)-quasi-idempotent. Let~\( \delta\defeq\norm*{e-f} \).
Assume that
\[
    \max\{\varepsilon_e,\varepsilon_f\}+\frac{1}{4}\delta^2<\frac{1}{4},
\]
then~\( e \) and~\( f \) are homotopic as quasi-idempotents.    
\end{lemma}
\begin{proof}
For each~\( s\in[0,1] \), we have
\begin{align*} 
&\big((1-s)e+sf\big)^2-\big((1-s)e+sf\big)\\
=&(1-s)(e^2-e)+(s^2-s)e^2+s(f^2-f)+(s^2-s)f^2+(s-s^2)(ef+fe) \\
=&(1-s)(e^2-e)+s(f^2-f)+(s^2-s)(e-f)^2.
\end{align*}
Thus
\[ 
\norm*{\big((1-s)e+sf\big)^2-\big((1-s)e+sf\big)}<\max\{\varepsilon_e,\varepsilon_f\}+\frac{1}{4}\norm*{e-f}^2
<\frac{1}{4}.
\]
So~\( s\mapsto (1-s)e+sf \) is a homotopy of quasi-idempotents between~\( e
\) and~\( f \).    
\end{proof}

Combining the two preceding lemmas, we have the following
\begin{theorem}\label{thm:close_quasi-idempotents_homotopic}
Let~\( A \) be a \Cst-algebra and~\( e,f\in A \). Let~\( \delta\defeq
\norm*{e-f} \).   
\begin{enumerate}
\item If~\( e \) is an~\( \varepsilon \)-quasi-idempotent and
\( f \) satisfies
\[ 
    (2\norm*{e}+\frac{5}{4}\delta+1)\delta+\varepsilon<\frac{1}{4}.
\]
Then~\( f \) is a~\( ((2\norm*{e}+\delta+1)\delta+\varepsilon)
\)-quasi-idempotent, which is homotopic to~\( e \) via a
path of quasi-idempotents.
\item As a special case, if~\( e \) is a projection and~\( f\in A \) is
self-adjoint satisfying~\( \delta<\frac{1}{17} \).
Then~\( f \) is a~\( 4\delta \)-quasi-projection, which is homotopic to
\( e \) as quasi-projections.  
\end{enumerate}
\end{theorem}
\begin{proof}
Let~\( e \) be an~\( \varepsilon \)-quasi-idempotent. It follows from the
\Cref{lem:estimate_f2-f} that~\( f \) is an
\( \varepsilon_f \)-quasi-idempotent, where 
\[ 
    \varepsilon_f\defeq (2\norm*{e}+\delta+1)\delta+\varepsilon;
\]
and
\[ 
    \max\{\varepsilon_e,\varepsilon_f\}+\frac{1}{4}\delta^2<(2\norm*{e}+\frac{5}{4}\delta+1)\delta+\varepsilon<\frac{1}{4}.
\]
Then it follows from the assumption and
\Cref{lem:homotopy_quasi-idempotents} that~\( e \) and~\(
f\) are homotopic through the straightline path connecting them.

If~\( e \) is a projection and~\( \delta<\frac{1}{17} \), 
then~\( \varepsilon=0 \) and~\(
\norm*{e}=1 \). So 
\[ 
    (2\norm*{e}+\frac{5}{4}\delta+1)\delta+\varepsilon<(2+\frac{5}{4}+1)\delta<\frac{1}{4}.
\]
If~\( f \) is self-adjoint, then by (1), 
\( f \) is~\( 4\delta \)-quasi-projection which is homotopic to~\(
e\) via the straightline homotopy connecting~\( e \) and~\( f \) because
\[ 
    (2\norm{e}+\delta+1)\delta+\varepsilon<(2+1+1)\delta=4\delta.
\]
Since both~\( e \) and~\( f \) are self-adjoint, it follows that~\(
(1-s)e+sf \) is self-adjoint for each~\( s\in[0,1] \). This implies that
the~\( e \) and~\( f \) are homotopic through a path of self-adjoint
quasi-idempotents, i.e.~quasi-projections.       
\end{proof}

\subsection{Unbounded Kasparov modules}
We recall the definition of odd, unbounded Kasparov modules. In particular,
we explain how an unbounded Kasparov module determines
an \emph{extension class}, that is, a class in~\( \KK_1(A,B) \) represented
by an extension. This is required for Connes--Higson's construction of 
asymptotic morphism in \Cref{lem:Connes-Higson_extension_to_E}.

Let~\( \mathcal{E} \) and~\( \mathcal{F} \)
be Hilbert~\( B \)-modules for a \Cst-algebra~\( B \), we write:
\begin{itemize}
\item \( \Cpt_B(\mathcal{E}) \) for the \Cst-algebra of compact operators
on~\( \mathcal{E} \); 
\item \( \Fin^*_B(\mathcal{E}) \) for the algebra of all finite-rank
operators on~\( \mathcal{E}\);
\item \( \End^*_B(\mathcal{E}) \) for the \Cst-algebras of bounded
adjointable operators on~\( \mathcal{E} \);
\item \( \Hom^*_B(\mathcal{E},\mathcal{F}) \) for the space of all~\( B
\)-linear, bounded adjointable operators~\( \mathcal{E}\to \mathcal{F} \).
\end{itemize}
\begin{definition}
Let~\( A \) and~\( B \) be \Cst-algebras. An \emph{odd unbounded Kasparov
\( A \)-\( B \)--module} is a triple~\( (\mathcal{A},E,D) \), where: 
\begin{itemize} 
\item \( E \) is a Hilbert~\( B \)-module;
\item \( \varrho\colon A\to\End^*_B(E) \) is an \emph{essential}
\st-homomorphism, that is,
\( \overline{\varrho(A)E}=E \); and
\( \mathcal{A}\subseteq A \) is a dense \st-subalgebra; we shall always
abbreviate~\( \varrho(a) \) as~\( a \).  
\item \( D\colon\dom D\subseteq E\to E \) 
is a self-adjoint and regular operator.
\end{itemize} 
such that:
\begin{itemize} 
\item \( \varrho(a)(D+\mathrm{i})^{-1}\in\Cpt_B(E) \) 
for all~\( a\in\mathcal{A} \) (and hence for
all~\( a\in A \));
\item For every~\( a\in\mathcal{A} \),~\( \varrho(a) \) maps~\( \dom D \)
into~\( \dom D \), and~\( [D,\varrho(a)] \)
extends to an element of~\( \End^*_B(E) \).
\end{itemize}
\end{definition}

\begin{remark}
We note that in the common definition of odd unbounded Kasparov modules,
\( \varrho\colon A\to\End_B^*(E) \) is \emph{not} assumed to be
essential. However, this turns out to be necessary for several 
constructions of the present paper, in particular, in
\Cref{app:asymptotic_morphism}. Nevertheless, every class in~\(
\KK_1(A,B) \) can be represented by an essential odd unbounded
Kasparov module.

If~\( A \) is unital, then 
an essential \st-homomorphism~\( \varrho\colon A\to\End_B^*(E) \) is
always unital. This implies that~\(
(D+\mathrm{i})^{-1}=\varrho(1)(D+\mathrm{i})^{-1}\in\Cpt_B(E) \). If
moreover~\( B=\C \), then~\( D \) has compact resolvent, hence has discrete
spectrum.
\end{remark}

We follow the line of \cite{Kasparov:K-functor_extension}*{Section 7} to
construct an extension of \Cst-algebras from an odd bounded Kasparov
module. Let~\( \chi\colon \R\to[-1,1] \) be a chopping function, i.e.~\(
\lim_{x\to\pm \infty}\chi(x)=\pm 1 \), then 
the \emph{bounded transform}
\[ 
    (\mathcal{A},E,D)\quad\longmapsto\quad (A,E,\chi(D))
\]
gives a bounded Kasparov module 
(cf.~\cite{Baaj-Julg:Unbounded_Kasparov_theory}*{Proposition 2.2} for the
special case~\( \chi(x)\defeq x(1+x^2)^{-1/2} \)). 
Then the unbounded Kasparov module~\(
(\mathcal{A},E,D) \) gives rise to a Busby invariant 
\[ \beta\colon A\to\End_B^*(E)/\Cpt_B(E),\quad
a\mapsto q(P_1 aP_1) \]
where~\( P_1\defeq \frac{1}{2}(\chi(D)+1) \) is as in \eqref{eq:P_1}, and 
\( q\colon\End^*_B(E)\to \End^*_B(E)/\Cpt_B(E) \) is the quotient map. 
This Busby invariant 
corresponds to the semi-split extension 
\begin{equation}\label{eq:Kasparov_KK_to_extension} 
\begin{tikzcd}
\Cpt_B(E) \arrow[r, tail] & M
\arrow[r, two heads] & A \arrow[l,bend left, "s"] 
\end{tikzcd} 
\end{equation} 
with~\( M \) the \Cst-algebra
\[
M\defeq \set*{(a,T)\in A\oplus
\End_B^*(E)\;\middle|\;\beta(\varrho(a))=q(T)} \]
with the semi-split section~\( s\colon A\to M \) given by
\( s(a)=(a,P_1\varrho(a)P_1) \).

The odd unbounded Kasparov module~\( (\mathcal{A},E,D) \) represents a
class in~\( [D]\in\KK_1(A,B) \). Via the Kasparov product
\eqref{eq:KK_index_pairing}, 
it gives a map
\[ 
    \K_1(A)\xrightarrow{\times[D]}\K_0(B),
\]
which can now be rephrased as the following composition of maps:
\begin{equation}\label{eq:K-theoretic_index_pairing}
\K_1(A)\xrightarrow{\partial}\K_0(\Cpt_B(E))\xrightarrow{[E]}\K_0(B),
\end{equation}
where~\( \partial \) is the boundary map of the extension
\eqref{eq:Kasparov_KK_to_extension}, and 
the map~\( [E]\colon \K_0(\Cpt_B(E))\to\K_0(B) \) is induced by the proper
\Cst--\( \Cpt_B(E) \)-\( B \)--correspondence~\( E \). If~\( E \) is a
\emph{full} Hilbert~\( B \)-module, then~\( E \) becomes an imprimitivity
bimodule between~\( \Cpt_B(E) \) and~\( B \), and the induced map~\( [E] \)
is an isomorphism.

In the special case~\( B=\C \), then an odd unbounded Kaspaorv~\( A \)-\(
B\)--module is the same thing as an odd spectral triple~\(
(\mathcal{A},\mathcal{H},D) \). Every Hilbert space~\( \mathcal{H} \) 
is full as a Hilbert~\( \C \)-module, thus implements a Morita--Rieffel 
equivalence between~\( \Cpt(\mathcal{H}) \) and~\( \C \).

In the rest of this article, we shall write~\( \braket{[v],[D]} \) for the
image of~\( [v]\in\K_1(A) \) under the map~\( \partial\colon
\K_1(A)\to\K_0(\Cpt_B(E)) \), and call it the
\emph{\textup{(}\K-theoretic\textup{)} index pairing}. The Kasparov product
between the class~\( [v]\in\K_1(A) \) and~\( [D]\in\KK_1(A,B) \) will be
denoted by~\( [v]\times[D] \).   

\subsection{Asymptotic morphisms}
Let~\( A \) and~\( B \) be \Cst-algebras. 
An \emph{asymptotic morphism} from~\( A \) to~\( B \)
is a family of maps~\( \Phi_t\colon
A\to B \), parametrised by~\( t\in [1,\infty) \), that ``asymptotically''
becomes a \st-homomorphism. Such an asymptotic morphism defines an element
in the E-theory group~\( \E(A,B) \), which via the E-theory product
\begin{equation}\label{eq:E_index_pairing}\tag{E}
    \underbrace{\E_i(\C,A)}_{\simeq\K_i(A)}
    \times\E_j(A,B)\to\underbrace{\E_{i+j}(\C,B)}_{\simeq\K_{i+j}(B)},
\end{equation}
induces a group homomorphism~\( \K_0(A)\to\K_0(B) \).

We recall form~\cite{Connes-Higson:E-theory} 
some basic facts of asymptotic morphisms. 
We fix the following notation: let~\( (a_t)_{t\in[1,\infty)} \) 
and~\( (b_t)_{t\in[1,\infty)} \) be two nets of elements in
the same \Cst-algebra, then we write
\[ 
    a_t\sim b_t\qquad\text{iff}\qquad
    \lim_{t\to \infty}\norm*{a_t-b_t}=0.  
\]

\begin{definition}\label{def:asymptotic_morphism}
Let~\( A \) and~\( B \) be \Cst-algebras. An \emph{asymptotic morphism} from~\( A \) to
\( B \) consists of a family of maps~\( (\Phi_t\colon A\to
B)_{t\in[1,\infty)} \) such that:
\begin{enumerate}
\item \( t\mapsto \Phi_t(a) \) is continuous for all~\( a\in A \);
\item for all
\( a,a'\in A \) and~\( \lambda\in\C \),
\[ \Phi_t(a^*)\sim\Phi_t(a)^*,\quad
\Phi_t(aa')\sim\Phi_t(a)\Phi_t(a'),\quad \Phi_t(a+\lambda
a')\sim\Phi_t(a)+\lambda\Phi_t(a'). \]
\end{enumerate}
\end{definition}

\begin{definition}
Two asymptotic morphisms~\( (\Phi_t\colon A\to B) \) and
\( (\Psi_t\colon A\to B) \) are \emph{\textup{(}asymptotically\textup{)}
equivalent}, if for all~\( a\in A \), it follows that 
\(\Phi_t(a)\sim\Psi_t(a)\).

Two asymptotic morphisms~\( (\Phi^0_t\colon A\to B) \) and~\( (\Phi^1_t\colon
A\to B) \) are \emph{\textup{(}asymptotically\textup{)} homotopic}, 
if there exists an asymptotic morphism
\( \Phi_t\colon A\to\Cont([0,1],B) \) such that
\( \ev_i\circ\Phi_t\sim\Phi_t^i \) for~\( i=0,1 \).
\end{definition}

\begin{remark}
If two asymptotic morphisms are equivalent, then they are homotopic via a
straight line connecting them.
\end{remark}

\begin{definition}[Extension to matrix algebras and unitisations]
Let~\( (\Phi_t\colon A\to B) \) be an asymptotic morphism, then for each~\(
n\in\N\),~\( (\Phi_t) \)  extends
to an asymptotic morphism from~\( \Mat_n(A^+) \) to~\( \Mat_n(B^+) \)
(still denoted~\( \Phi_t \) for simplicity)
as follows: one first extends~\( \Phi_t \) to the minimal 
unitisations of~\( A \) and~\( B \)  by setting
\[ 
    \Phi_t(a,\lambda)\defeq (\Phi_t^+(a),\lambda),
\]
then extend to their matrix algebras by applying~\( \Phi_t \) entrywise. 
\end{definition}

The following result is well-known,
cf.~\cite{Guentner-Higson-Trout:Equivariant_E-theory}*{Chapter 1}:

\begin{proposition}
Let~\( (\Phi_t\colon A\to B) \) be an asymptotic morphism, then it induces a
map~\( \Phi_*\colon \K_0(A)\to\K_0(B) \) as follows: let~\( p\in\Mat_n(A^+)
\) be a projection. By definition, for sufficiently large~\( t_0>0 \), there
exists a projection~\( q_t \) that is close to~\( \Phi_t(p) \) for each~\(
t>t_0\), and all of these~\( q_t \)'s for~\( t>t_0 \)  
are homotopic. This gives a well-defined class in~\( \K_0(B) \).     
\end{proposition}

If we represent~\( \K_0 \)-groups by quasi-idempotents, then the preceding
proposition can be rephrased as: for each projection~\( p\in \Mat_n(A^+)
\), there exists~\( t_0>0 \), such that~\( [\Phi_t(p)] \) gives the same
class in~\( \K_0(B) \) for every~\( t>t_0 \), which is defined to be the
image of~\( [p] \) under the K-theory induced map of~\( (\Phi_t)_t \).      
We note in particular that, if~\( \Phi_t\colon A\to B \) and~\( \Psi_t\colon
A\to B \) are \emph{homotopic} asymptotic morphisms, then~\( \Phi_*\colon
\K_0(A)\to\K_0(B) \) coincides with~\( \Psi_*\colon \K_0(A)\to\K_0(B) \). 

\begin{proposition}
The asymptotic morphism
\begin{equation}\label{eq:Phi^D}
\Phi^D_t\colon \Co(\R)\otimes A\to\Cpt_B(E),\quad f\otimes a\mapsto
f(t^{-1}D)a
\end{equation}
induces a map~\( \K_0(\mathrm{S}A)\to\K_0(B) \),
which coincides with the
index pairing~\( \braket{[v],[D]} \) under the isomorphism~\(
\K_1(A)\simeq\K_0(\mathrm{S}A) \) in \Cref{lem:odd_K-theory}. 
Here~\( \mathrm{S}A \) refers to the
non-unital \Cst-algebra~\( \Co(\R)\otimes A \). 
\end{proposition}

\begin{proof}
We provide the proof in \Cref{app:asymptotic_morphism} based on the
following facts: given an unbounded Kasparov module 
\( (\mathcal{A},E,D) \), it gives an extension
of \Cst-algebras as in \eqref{eq:Kasparov_KK_to_extension} and hence an
asymptotic morphism~\( (\Phi^\mathrm{CH}_t) \)
via the Connes--Higson's construction
\[
\Phi^\mathrm{CH}_t\colon \Co(0,1)\otimes B\to I,\qquad f\otimes b\mapsto
f(u_t)s(b).
\]
Then~\( (\Phi^\mathrm{CH}_t) \) induces the K-theoretic index pairing
\[ 
    \partial\colon\K_0(\mathrm{S}A)\to\K_0(\Cpt_B(E))
\]
associated to the unbounded Kasparov module~\( (\mathcal{A},E,D) \) as in
\eqref{eq:K-theoretic_index_pairing}. It is invariant under
\emph{homotopies} of asymptotic morphisms. Thus, it suffices to show that 
\eqref{eq:Phi^D} is, up to identifying~\( \R \) with~\( (0,1) \) using a
homeomorphism,  
homotopic to the Connes--Higson asymptotic 
morphism~\( (\Phi^{\mathrm{CH}}_t)_t \). 
\end{proof}

\section{Index pairing via asymptotic morphisms}
\subsection{The odd K-theory class}
An odd unbounded Kasparov~\( A \)-\( B \)--module~\( (\mathcal{A},E,D) \)
gives rise to a semi-split extension \eqref{eq:Kasparov_KK_to_extension},
henceforce an asymptotic morphism~\( \Co(\R)\otimes A\to\Cpt_B(E) \). We
shall compute the image of the class~\( [v]\in\K_1(A) \) of a unitary~\(
v\in A \) under this asymptotic morphism. In order to do this,
we must first map the class~\(
[v]\in\K_1(A) \) of a unitary~\( v\in A \) to a class in~\(
\K_0(\mathrm{S}A) \) under the suspension isomorphism~\(
\K_1(A)\simeq\K_0(\mathrm{S}A) \).

\begin{lemma}\label{lem:odd_K-theory}
Let~\( v\in A \) be a unitary. Let~\( c,s\colon \R\to[0,1] \) be continuous
functions that satisfy
\[
    c^2(x)+s^2(x)=1,\quad \lim_{x\to\infty}s(x)=1,\quad
    \lim_{x\to-\infty}c(x)=1.
\]
Define~\( e,f\colon \R\to\Mat_2(A) \) as 
\begin{gather*} 
    e(x)\defeq
    \begin{pmatrix}
    s(x)^2\otimes1 & c(x)s(x)\otimes v \\
    c(x)s(x)\otimes v^* & c(x)^2\otimes 1
    \end{pmatrix};\quad
    f(x)\defeq
    \begin{pmatrix}
    0 & 0 \\ 0 & 1
    \end{pmatrix}.
\end{gather*}
Then both~\( e \) and~\( f \)  
are projections in~\( \Mat_2(\mathrm{S}A^+) \), and
\( [e]-[f]\in\K_0(\mathrm{S}A) \) is the image of~\(
[v]\in\K_1(A) \) under the suspension isomorphism~\(
\K_1(A)\xrightarrow{\sim}\K_0(\mathrm{S}A) \).  
\end{lemma}

\begin{proof}
Let
\[ \begin{tikzcd}
\mathrm{S}A\arrow[r, tail] & \mathrm{C}A \arrow[r, two heads,
"\ev_{+\infty}"] & A
\end{tikzcd} \]
be the suspension--cone extension, where~\( \mathrm{C}A\defeq
\Co((-\infty,+\infty],A) \) and~\( \ev_{+\infty}\colon \mathrm{C}A\to A \) is the evalutation map 
at~\( +\infty \).
The boundary map associated to this extension
\( \partial\colon\K_1(A)\to\K_0(\mathrm{S}A) \) is an isomorphism because
\( \mathrm{C}A \) is contractible. Up to a sign, this uniquely determines
the suspension isomorphism~\( \K_1(A)\simeq\K_0(\mathrm{S}A) \). 

We use the partial isometry picture to compute the suspension
isomorphism as in
\cite{Roerdam-Larsen-Laustsen:K-theory}*{Proposition 9.2.2}. Let~\(
z\colon\R\to\Mat_2(A) \) be defined as
\[ 
    z\defeq\begin{pmatrix}
    s(x)\otimes v & -c(x)\otimes 1
    \\
    0 & 0
    \end{pmatrix}.  
\]
Then~\( z\in\Mat_2(\mathrm{C}A^+) \) is a partial isometry, which lifts~\(
v\oplus 0 \) under~\( \ev_{+\infty} \). Thus we have
\[ 
    \partial[v]=[\mathbf{1}_{2\times 2}-z^*z]-[\mathbf{1}_{2\times
    2}-zz^*].
\]
We compute that
\begin{align*} 
\mathbf{1}_{2\times 2}-z^*z&=
\begin{pmatrix}
1 & 0 \\ 0 & 1
\end{pmatrix}-
\begin{pmatrix}
s^2(x)\otimes 1 &
-c(x)s(x)\otimes v^* \\
-c(x)s(x)\otimes v &
c^2(x)\otimes 1 
\end{pmatrix}\\
&=\begin{pmatrix}
c^2(x)\otimes 1 &
c(x)s(x)\otimes v^* \\
c(x)s(x)\otimes v &
s^2(x)\otimes 1 
\end{pmatrix}; \\
\mathbf{1}_{2\times 2}-zz^*&=
\begin{pmatrix}
1 & 0 \\ 0 & 1
\end{pmatrix}-\begin{pmatrix}
1 & 0 \\ 0 & 0
\end{pmatrix}
=\begin{pmatrix}
0 & 0 \\ 0 & 1
\end{pmatrix}=f.
\end{align*}
The claim follows as
\[ 
    e=\begin{pmatrix}
    0 & 1 \\ 1 & 0
    \end{pmatrix}\begin{pmatrix}
    c^2(x)\otimes 1 & c(x)s(x)\otimes v^* \\
    c(x)s(x)\otimes v & s^2(x)\otimes 1
    \end{pmatrix}
    \begin{pmatrix}
    0 & 1 \\ 1 & 0
    \end{pmatrix}.\qedhere
\]
\end{proof}
\begin{remark}
We note that the natural isomorphism 
\( \K_1(A)\simeq\K_0(\mathrm{S}A) \) is unique only up to a sign, as shown
by Elliott \cite{Elliott:Uniqueness_index_map_complex}. For
instance, if one replaces~\( \mathrm{C}A=\Co((-\infty,+\infty],A) \) by~\(
\Co([-\infty,+\infty),A) \) and~\( \ev_{+\infty} \) by~\( \ev_{-\infty} \),
then the same index map formula as in
\cite{Roerdam-Larsen-Laustsen:K-theory} yields
the inverse of the class~\( [e]-[f]\in\K_0(\mathrm{S}A) \) 
constructed above. The convention that we have adopted gives the correct
index pairing, 
cf.~\Cref{thm:index_pairing_quasi-projections,thm:index_pairing_submodule}.
\end{remark}

\subsection{Representing the odd index pairing by quasi-projections}
We apply the asymptotic morphism \eqref{eq:Phi^D} to the K-theory class
\( [e]-[f]\in\K_0(\mathrm{S}A) \). Set~\(c_t\defeq c(t^{-1}D) \) and 
\( s_t\defeq s(t^{-1}D) \);
define~\( \check{e}_t\defeq \Phi^D_t(e) \) and 
\( f_t\defeq \Phi^D_t(f) \). Then  
\[ 
\check{e}_t=\begin{pmatrix}
    s_t^2 & c_ts_tv \\
    c_ts_tv^* & c^2_t
    \end{pmatrix},\quad
f_t=\begin{pmatrix}
0 & 0 \\ 0 & 1
\end{pmatrix}.
\] 

Thus~\( (f_t)_{t\in [1,\infty)} \) 
is a constant family of projection.
The family of operators~\( (\check{e}_t)_{t\in [1,\infty)} \) satisfies
\[ 
    \check{e}_t^*-\check{e}_t\sim 0,\quad \check{e}_t^2-\check{e}_t\sim 0
\]
and hence ``asymptotically'' a projection. For sufficiently large~\( t \),
\( \check{e}_t \) is a quasi-idempotent, yet not self-adjoint. The
asymptotic self-adjointness, nevertheless, allows us to replace~\(
(\check{e}_t)_{t\in [1,\infty)} \) by a family~\( (e_t)_{t\in [1,\infty)} \)
of \emph{quasi-projections}, up to homotopies of
quasi-idempotents. To obtain a quantitative estimate as in
the proposition below, we need extra assumptions for the commutators
\[ 
    [v,c_t],\quad [v,s_t],\quad [v,\sqrt{c_ts_t}].
\]

\begin{proposition}
Let
\begin{equation}\label{eq:e_t}
    e_t\defeq\begin{pmatrix}
     s_t^2 & \sqrt{c_ts_t} v\sqrt{c_ts_t} \\
    \sqrt{c_ts_t} v^*\sqrt{c_ts_t} &  c_t^2
    \end{pmatrix},
\end{equation}
Assume that there exists a constant~\( R>0 \),
independent of~\( t \),  such that
\begin{equation}\label{eq:assmption_constants_R}
\left.\begin{array}{r}
\norm*{[c_t,v]} \\
\norm*{[s_t,v]} \\
\norm*{[\sqrt{c_ts_t},v]}
\end{array}\right\}\leq R\cdot t^{-1}\norm*{[D,v]}
\end{equation}
Let
\[ 
0<\varepsilon<\frac{8}{237}(\sqrt{1393}-34)\approx 0.1122. 
\]
Then for every~\( t>2\varepsilon^{-1}R\norm*{[D,v]} \), the following
holds\textup{:}
\begin{enumerate}
\item \( e_t \) is an~\( \varepsilon \)-quasi-projection.
\item \( \norm*{e_t-\check{e}_t}<\frac{\sqrt{2}}{4}\varepsilon \). 
\item \( e_t \) is homotopic to~\( \check{e}_t \) via a
path of quasi-idempotents.
\end{enumerate}
\end{proposition}

\begin{proof}
To prove (1), we compute that
\begin{align*} 
e_t^2&=\begin{pmatrix}
s_t^2 & \sqrt{c_ts_t} v\sqrt{c_ts_t} \\
\sqrt{c_ts_t} v^*\sqrt{c_ts_t} &  c_t^2
\end{pmatrix}\begin{pmatrix}
s_t^2 & \sqrt{c_ts_t} v\sqrt{c_ts_t} \\
\sqrt{c_ts_t} v^*\sqrt{c_ts_t} &  c_t^2
\end{pmatrix} \\
&=\begin{pmatrix}
s_t^2+\sqrt{c_ts_t}[v,c_ts_t]v^*\sqrt{c_ts_t} & 
\sqrt{c_ts_t}v\sqrt{c_ts_t}+[\sqrt{c_ts_t}v\sqrt{c_ts_t},c_t^2] \\
\sqrt{c_ts_t}v^*\sqrt{c_ts_t}+[c_t^2,\sqrt{c_ts_t}v^*\sqrt{c_ts_t}] &
c_t^2+\sqrt{c_ts_t}v^*[c_ts_t,v]\sqrt{c_ts_t}
\end{pmatrix}.
\end{align*}
So
\begin{align*}
e_t^2-e_t=&\begin{pmatrix}
\sqrt{c_ts_t}[v,c_ts_t]v^*\sqrt{c_ts_t} & 0 \\ 0 &
\sqrt{c_ts_t}v^*[c_ts_t,v]\sqrt{c_ts_t}
\end{pmatrix}\\
&+\begin{pmatrix}
0 & [\sqrt{c_ts_t}v\sqrt{c_ts_t},c_t^2] \\
[c_t^2,\sqrt{c_ts_t}v^*\sqrt{c_ts_t}] & 0
\end{pmatrix}.
\end{align*}
Therefore,
\begin{align*}
\norm{e_t^2-e_t}\leq&
\max\{\norm*{\sqrt{c_ts_t}[v,c_ts_t]v^*\sqrt{c_ts_t}},\norm*{\sqrt{c_ts_t}v^*[c_ts_t,v]\sqrt{c_ts_t}}\}\\
&+\norm*{[\sqrt{c_ts_t}v\sqrt{c_ts_t},c_t^2]}
\end{align*}
Now we note that~\( v^*,v,\sqrt{c_t},\sqrt{s_t} \) are all contractive
operators, and
\[
    0\leq c_ts_t\leq\frac{1}{2}\left(c_t^2+s_t^2\right)=\frac{1}{2}\cdot\id.
\]
So~\( \sqrt{c_ts_t} \) has norm bounded by~\( \frac{\sqrt{2}}{2} \). 
By assumption \eqref{eq:assmption_constants_R}, we obtain
\begin{alignat*}{3}
\norm*{\sqrt{c_ts_t}[v,c_ts_t]v^*\sqrt{c_ts_t}}&\leq
(\tfrac{\sqrt{2}}{2})^2\norm*{[v,c_t]}+(\tfrac{\sqrt{2}}{2})^2\norm*{[v,s_t]}&\leq
R\cdot t^{-1}\norm*{[D,v^*]};\\
\norm*{\sqrt{c_ts_t}v^*[c_ts_t,v]\sqrt{c_ts_t}}&\leq
(\tfrac{\sqrt{2}}{2})^2\norm*{[c_t,v]}+(\tfrac{\sqrt{2}}{2})^2\norm*{[s_t,v]}&\leq
R\cdot t^{-1}\norm*{[D,v^*]};\\
\norm*{[\sqrt{c_ts_t}v\sqrt{c_ts_t},c_t^2]}&\leq
(\tfrac{\sqrt{2}}{2})^2\norm*{[v,c_t]}+(\tfrac{\sqrt{2}}{2})^2\norm*{[v,c_t]}&
\leq R\cdot t^{-1}\norm*{[D,v^*]}.
\end{alignat*}
Therefore,~\( \norm*{e_t^2-e_t}\leq 2R\cdot t^{-1}\norm*{[D,v^*]}<\varepsilon \). 
Clearly~\( e_t
\) is self-adjoint. This finishes the proof of (1).

Now we prove (2).
Use again the assumption \eqref{eq:assmption_constants_R}, we have
\begin{align*} 
\norm*{e_t-\check{e}_t}&=\norm*{\begin{pmatrix}
    0 & \sqrt{c_ts_t}[v,\sqrt{c_ts_t}] \\
    \sqrt{c_ts_t}[v^*,\sqrt{c_ts_t}] & 0
\end{pmatrix}}\\
&\leq\norm*{\sqrt{c_ts_t}}\cdot\norm*{[v,\sqrt{c_ts_t}]}\leq
\frac{\sqrt{2}}{2}Rt^{-1}\norm*{[D,v]}<\frac{\sqrt{2}}{4}\varepsilon.
\end{align*}

Now we prove (3). Note
that~\( e_t \) is an~\( \varepsilon \)-quasi-projection, so~\(
\norm*{e_t}<1+\varepsilon \); and~\(
\delta\defeq\norm*{\check{e}_t-e_t}<\frac{\sqrt{2}}{4}\varepsilon \). 
It follows that (below we take~\( \sqrt{2}<1.5 \)):
\begin{align*} 
\left(2\norm*{e_t}+\frac{5}{4}\delta+1\right)\delta+\varepsilon&<
\left(2(1+\varepsilon)+\frac{5}{4}\cdot\frac{\sqrt{2}}{4}\varepsilon+1\right)\frac{\sqrt{2}}{4}\varepsilon+\varepsilon\\
&<\left(\frac{79}{32}\varepsilon+3\right)\cdot\frac{3}{8}\varepsilon+\varepsilon \\
&=\frac{237}{256}\varepsilon^2+\frac{17}{8}\varepsilon<\frac{1}{4}.
\end{align*}
Therefore, by \Cref{thm:close_quasi-idempotents_homotopic},~\( e_t \) and
\( \check{e}_t \) are homotopic via a path of quasi-idempotents. 
\end{proof}

Then \Cref{prop:K-theory_quasi-idempotents} leads to the following
result:

\begin{theorem}\label{thm:index_pairing_quasi-projections}
Let~\( \varepsilon<0.1121 \). Assume that there exists~\( R>0 \) such that
\eqref{eq:assmption_constants_R} holds.
Then the 
\K-theoretic index pairing~\( \braket{[v],[D]} \) is represented by
the formal difference of \emph{(quasi-)projections}
\[ 
  [e_t]-[f_t]=
  \left[\begin{pmatrix}
   s_t^2 & \sqrt{c_ts_t} v\sqrt{c_ts_t} \\
  \sqrt{c_ts_t} v^*\sqrt{c_ts_t} &  c_t^2
  \end{pmatrix}\right]-
  \left[\begin{pmatrix}
  0 & 0 \\ 0 & 1
  \end{pmatrix}\right]
\]
for any choice of~\( t\in [1,\infty) \) satisfying 
\( t>2\varepsilon^{-1}R\norm*{[D,v]} \). 
\end{theorem}

\begin{remark}
It is possible to improve the upper bound~\( 0.1121 \) for 
\( \varepsilon \) with a more accurate estimate of the
norm~\( \norm*{e_t-\check{e}_t} \), together with a finer assumption that
treats the constant~\( R \) for~\( c_t \),~\( s_t \) and~\( \sqrt{c_ts_t}
\) separately. However, we note that this is of little help for
constructing finite-dimensional K-theory representatives in the next
section, as the control of~\( \varepsilon \) will be dominated by
the estimate~\( \varepsilon+\delta<\frac{1}{400}=0.0025 \) in
\Cref{thm:index_pairing_submodule}.
\end{remark}

\section{Spectral truncation of the index pairing}
We have constructed from the previous section, a pair of
\emph{quasi-projections}~\( e_t \) and~\( f_t \), such that the K-theoretic
index pairing~\( \braket{[v],[D]} \) is given by
\[ 
    [e_t]-[f_t]\in\K_0(\Cpt_B(E)).
\]
The next goal is to apply a 
\emph{spectral truncation} to them. Namely, we shall
replace~\( e_t \) and~\( f_t \) by quasi-projections~\( p^e_{t,\lambda} \)
and~\( p^f_{t,\lambda} \) on a
submodule of~\( E\oplus E \), which 
corresponds to the ``low-lying spectrum'' of~\( D \). The submodule depends
on a sufficiently large choice of the parameter 
\( \lambda>0 \), which occurs 
as the threshold of this truncated spectrum of~\(D\). 
The existence of such a submodule is, however, not always assured. It
always exists if~\( B=\C \) and~\( E \) is a Hilbert space.  

We shall show that the quasi-projections~\( e_t \) and~\(
f_t\) are homotopic to direct sums of quasi-projections,
\[ 
    p^e_{t,\lambda}\oplus q^e_{t,\lambda},\quad p^f_{t,\lambda}\oplus
    q^f_{t,\lambda},
\]
also~\( q^e_{t,\lambda} \) is homotopic to~\( q^f_{t,\lambda} \). All these
homotopies are through quasi-projections in~\( \Mat_2(\Cpt_B(E)^+)
\). Therefore, it follows from the description of~\( \K_0(\Cpt_B(E)) \) by
quasi-idempotents or quasi-projections
(\Cref{prop:K-theory_quasi-idempotents}) that
\[ 
    [e_t]-[f_t]=[p^e_{t,\lambda}]+[q^e_{t,\lambda}]-[p^f_{t,\lambda}]-[q^f_{t,\lambda}]=[p^e_{t,\lambda}]-[p^f_{t,\lambda}].
\]

In this section, we shall prove that 
the K-theory representative~\( [p^e_{t,\lambda}]-[p^f_{t,\lambda}] \)
gives rise to 
the odd spectral localiser~\( L_{\kappa,\lambda} \) 
of Loring and Schulz-Baldes when~\( B=\C \). 
More precisely, we will show that the
spectral localiser~\( L_{\kappa,\lambda} \) can be recovered from a particular
choice of the functions~\( s \) and~\( c \), and that  
\[ 
    [e_t]-[f_t]=[p^e_{t,\lambda}]-[p^f_{t,\lambda}]=
    \frac{1}{2}\sig(L_{t^{-1},\lambda}).
\]
Thus our result gives a new proof of the main theorem in
\cites{Loring-SBaldes:Odd_spectral_localiser,Loring-SBaldes:Odd_spectral_localiser_sf},
as well as provides a new interpretation of the spectral localiser.

We also consider a special case of the \emph{numerical index pairing},
given by pairing the Kasparov product~\( [v]\times[D] \) with a finite
trace~\( \tau\colon B\to\C \). We show that the spectral truncation method
still applies, providing a representative of the numerical index pairing
that is supported on a spectral subspace of the localised Hilbert space~\(
\mathcal{H}^\tau\defeq E\otimes_BL^2(E,\tau) \). This recovers the
semi-finite spectral localiser of Schulz-Baldes and Stoiber
\cite{SBaldes-Stoiber:Semifinite_spectral_localizer}, and gives a new proof
of their main theorem in this special setting.

\subsection{Spectral decomposition of Hilbert \Cst-modules}
We introduce the following definition for studying spectral truncations of
index pairings.

\begin{definition}\label{def:spectral_decomposition}
Let~\( B \) be a \Cst-algebra,~\( \mathcal{E} \) be a Hilbert~\( B
\)-module, and~\( \mathcal{D} \) be an unbounded, self-adjoint regular operator
on~\( \mathcal{E} \). Given~\( \lambda>0 \),
A \emph{spectral decomposition} of~\( \mathcal{E} \)
for the operator~\( \mathcal{D} \) with \emph{spectral threshold}~\( \lambda \) 
is a pair of mutually \emph{complemented} submodules
\( (\mathcal{E}_{\lambda}^{\downarrow},\mathcal{E}_{\lambda}^{\uparrow}) \) 
of~\( \mathcal{E} \), such that the operator inequality
\[ 
    \braket{\mathcal{D}\xi,\mathcal{D}\xi}\geq\lambda^2\braket{\xi,\xi}
\]
holds for all~\( \xi\in\dom \mathcal{D}\cap \mathcal{E}^\uparrow_\lambda \). 
We call~\( \mathcal{E}_\lambda^{\downarrow} \) 
(resp.~\( \mathcal{E}_\lambda^{\uparrow} \)) the
\emph{lower} (resp.~\emph{upper}) \emph{spectral} submodule for the
decomposition~\( \mathcal{E}\simeq \mathcal{E}_\lambda^{\downarrow}\oplus
\mathcal{E}_\lambda^{\uparrow} \).
\end{definition}

Let~\((\mathcal{E}_{\lambda}^\downarrow,\mathcal{E}_{\lambda}^\uparrow) \)
be a spectral decomposition of~\( \mathcal{E} \). Write~\(
P_{\lambda}^{\downarrow} \) and~\( P_{\lambda}^{\uparrow} \) for the
orthogonal projections onto ~\( \mathcal{E}_{\lambda}^{\downarrow} \) and
\( \mathcal{E}_{\lambda}^{\uparrow} \).  Then under the direct sum
decomposition
\[ 
    \mathcal{E}\simeq \mathcal{E}_{\lambda}^{\downarrow}\oplus \mathcal{E}_{\lambda}^{\uparrow},
\]
any operator~\( T \) on~\( \mathcal{E} \) decomposes as a~\( 2\times 2
\)-block operator: 
\begin{equation}\label{eq:spectral_decomposition_operator}
    \begin{pmatrix}
    p^T & n^T{} \\
    m^T & q^T
    \end{pmatrix},    
\end{equation}
where~\( p^T \),~\( q^T \),~\( m^T \) and~\( n^T \)  
are \emph{truncated} operators,
defined as    
\begin{alignat*}{3} 
p^T &\defeq P_{\lambda}^{\downarrow} TP_{\lambda}^{\downarrow} &&\colon \mathcal{E}_{\lambda}^{\downarrow}\to \mathcal{E}_{\lambda}^{\downarrow}, \\ q^T
&\defeq P_{\lambda}^{\uparrow}TP_{\lambda}^{\uparrow} &&\colon
\mathcal{E}_{\lambda}^{\uparrow}\to \mathcal{E}_{\lambda}^{\uparrow},\\ m^T &\defeq
P_{\lambda}^{\uparrow}TP_{\lambda}^{\downarrow}&&\colon \mathcal{E}_{\lambda}^{\downarrow}\to
\mathcal{E}_{\lambda}^{\uparrow}, \\ n^T &\defeq
P_{\lambda}^{\downarrow} TP_{\lambda}^{\uparrow}&&\colon \mathcal{E}_{\lambda}^{\uparrow}\to
\mathcal{E}_{\lambda}^{\downarrow}.
\end{alignat*}
If~\( T \) is self-adjoint, then clearly 
both~\( p^T \) and~\( q^T \) are self-adjoint, and~\( n^T=m^T{}^* \).

\begin{example}[Spectral decomposition of a Hilbert
space]\label{ex:spectral_decomposition_Hilbert_space}
We explain the terminology in the Hilbert space setting, i.e.~\( B=\C \)
and~\( \mathcal{E} \) is given by a Hilbert space~\( \mathcal{H} \).  Let~\(
\mathcal{D} \) be an unbounded, self-adjoint operator on~\( \mathcal{H} \).
Let~\( I\subseteq \R \) be a Borel subset and~\( \chi_I \) be its
characteristic function. Then Borel functional calculus gives a spectral
projection~\( \chi_I(\mathcal{D})\in\Bdd(\mathcal{H}) \). If~\( \mathcal{D}
\) has discrete spectrum, then~\( \mathcal{H}_I\defeq
\chi_I(\mathcal{D})\mathcal{H} \) is spanned by eigenvectors~\( \xi \)
of~\( \mathcal{D} \) with eigenvalue in~\( I \).  In particular, let~\( I
\) be any Borel subset of~\( (-\infty,-\lambda)\cup(\infty,\lambda) \),
then~\( (\mathcal{H}_{\R\setminus I},\mathcal{H}_I) \) is a spectral
decomposition of~\( \mathcal{H} \) for the operator~\( \mathcal{D} \) with
spectral threshold~\( \lambda \).
\end{example}

\begin{remark}
We remark that the 
existence of a spectral decomposition of a Hilbert \Cst-module
can not be guaranteed in general. 
Unlike the Hilbert space case, there is no Borel functional calculus
for operators on Hilbert \Cst-modules due to a lack of spectral measure. 
This prevents one from constructing spectral decompositions 
from characteristic
functions as in \Cref{ex:spectral_decomposition_Hilbert_space}.
We anticipate that, in order to obtain a more broadly applicable notion of
spectral decomposition of Hilbert \Cst-modules, 
the \emph{mutual orthogonality} of the submodules~\(
\mathcal{E}_\lambda^\downarrow \) and~\( \mathcal{E}_\lambda^\uparrow \)
needs be relaxed. 
\end{remark}

Later, we shall be concerned with spectral decomposition of the Hilbert
\( B \)-module~\( \widehat{E}\defeq E\oplus E \).   
We caution that there are be two ways to 
decomposition~\( \widehat{E} \) as mutually complemented submodules, 
one as~\( E\oplus E \) and another being the spectral decomposition~\(
E_{\lambda}^{\downarrow}\oplus E_{\lambda}^{\uparrow} \). Both
decompositions yield a~\( 2\times2 \)-matrix expression of an operator~\(
T \) on~\( \widehat{E} \). The above notation~\( p^T,q^T \) and~\(
m^T \) of \eqref{eq:spectral_decomposition_operator} 
will be used only for the spectral decomposition~\(
\widehat{E}\simeq E_{\lambda}^{\downarrow}\oplus
E_{\lambda}^{\uparrow} \). 

\subsection{Truncating quasi-projections to spectral submodules}
Let~\( (E_{\lambda}^\downarrow,E_{\lambda}^\uparrow) \) be a spectral
decomposition of~\( \widehat{E}\defeq E\oplus E \) for the operator~\(
D\oplus D \) with spectral threshold~\( \lambda \).  
In order to apply truncations to the quasi-projections~\( e_t \) and~\( f_t
\) to~\( E_{\lambda}^\downarrow \) 
without changing their representing K-theory classes, one first
homotopes them to a diagonal operator under the direct sum decomposition of
Hilbert modules~\( \widehat{E}=E_{\lambda}^{\downarrow}\oplus
E_{\lambda}^{\uparrow} \), based
on the following lemma. 

\begin{lemma}[Reduction to the diagonal]\label{lem:reduction_to_diagonal}
Let~\( \mathcal{E} \) be a Hilbert~\( B \)-module, and~\(
\mathcal{F}\subseteq\mathcal{E} \) be a
complemented submodule. 
Let~\( e\in\End^*_B(\mathcal{E}) \) be an~\( \varepsilon \)-quasi-projection. 
Under the direct sum decomposition of Hilbert~\( B \)-module~\(
\mathcal{E}\simeq\mathcal{F}\oplus\mathcal{F}^\perp \), write
\[ 
    e=\begin{pmatrix}
    p & m^* \\ m & q
    \end{pmatrix}  
\]
where~\( p\in\End^*_B(\mathcal{F}) \),~\(
m\in\Hom^*_B(\mathcal{F},\mathcal{F}^\perp) \) and~\(
q\in\End^*_B(\mathcal{F}^\perp) \).
Assume that there exists~\( \varepsilon_q<\frac{1}{400}-\varepsilon \) 
such that~\( q \) is an~\( \varepsilon_q \)-quasi-projection.
Then:
\begin{enumerate}
\item \( p \) is an~\( (2\varepsilon+\varepsilon_q) \)-quasi-projection;
\item \( e \) is homotopic to~\( \left(\begin{smallmatrix}
    p & 0 \\ 0 & q
\end{smallmatrix}\right) \) as quasi-projections.   
\end{enumerate}
\end{lemma}

\begin{proof}
Since~\( e \) is an~\( \varepsilon \)-quasi-projection, it follows that  
\begin{align*} 
\varepsilon&>\norm{e^2-e}=\norm*{\begin{pmatrix}
        p^2-p+m^*m & pm^*+m^*q-m^* \\
        mp+qm-m & mm^*+q^2-q
    \end{pmatrix}}\\
    &\geq\max\left\{\norm{p^2-p+m^*m},\norm{q^2-q+mm^*}\right\}.
\end{align*}
Also~\( q \) is an~\( \varepsilon_q \)-quasi-projection, so
\begin{gather*} 
\norm{m}^2=\norm{mm^*}\leq\varepsilon+\norm{q^2-q}<\varepsilon+\varepsilon_q,\\
\norm{p^2-p}<\norm{m^*m}+\varepsilon<
2\varepsilon+\varepsilon_q.
\end{gather*}
Now let~\( f\defeq
\left(\begin{smallmatrix}
    p & 0 \\ 0 & q
\end{smallmatrix}\right) \) and~\( \delta\defeq \norm{e-f} \).
Then we have
\[ 
    \delta=\norm*{\begin{pmatrix}
        0 & m^* \\ m & 0
    \end{pmatrix}
    }=\norm*{m}<\sqrt{\varepsilon+\varepsilon_q}.
\]
Therefore,
\begin{align*} 
\big(2\norm*{e}+\frac{5}{4}\delta+1\big)\delta+\varepsilon
&<\big(2(1+\varepsilon)+\frac{5}{4}\sqrt{\varepsilon+\varepsilon_q}+1\big)
\sqrt{\varepsilon+\varepsilon_q}+\varepsilon \\
&<\big(2+2\sqrt{\varepsilon+\varepsilon_q}+\frac{5}{4}\sqrt{\varepsilon+\varepsilon_q}+1+\sqrt{\varepsilon+\varepsilon_q}\big)\sqrt{\varepsilon+\varepsilon_q}
\\
&=3\sqrt{\varepsilon+\varepsilon_q}+\left(2+\frac{5}{4}+1\right)\left(\varepsilon+\varepsilon_q\right) \\
&<\left(3+\frac{17}{80}\right)\cdot\frac{1}{20}<\frac{1}{4}.
\end{align*}
It follows from \Cref{thm:close_quasi-idempotents_homotopic} that~\( f \)
is also a quasi-projection that is homotopic to~\( e \).
\end{proof}

Assume that~\( (E^\downarrow_{\lambda},E^\uparrow_{\lambda}) \) is
a spectral decomposition of~\( \widehat{E} \) for~\( D\oplus D \) with
spectral threshold~\( \lambda \). Decompose
\begin{equation}\label{eq:spectral_decomposition_et_ft}
    e_t=\begin{pmatrix} 
    p^e_{t,\lambda} & {m^e_{t,\lambda}}^* \\
    m^e_{t,\lambda} & q^e_{t,\lambda} 
    \end{pmatrix},\quad
    f_t=\begin{pmatrix} p^f_{t,\lambda} & {m^f_{t,\lambda}}^* \\
    m^f_{t,\lambda} & q^f_{t,\lambda} 
    \end{pmatrix}
\end{equation}
as in \eqref{eq:spectral_decomposition_operator},
i.e.~\(
p^e_{t,\lambda},p^f_{t,\lambda}\in\End_B^*(E_{\lambda}^{\downarrow})
\),~\(
m^e_{t,\lambda},m^f_{t,\lambda}\in\Hom_B^*(E_{\lambda}^{\downarrow},E_{\lambda}^{\uparrow})
\), and~\(
q^e_{t,\lambda},q^f_{t,\lambda}\in\End_B^*(E_{\lambda}^{\uparrow})
\). Recall that~\( f_t \) is given by the constant family of projections~\(
\left(\begin{smallmatrix}
    0 & 0 \\ 0 & \id
\end{smallmatrix}\right)\) on~\( \widehat{E} \). Therefore, we have
\[ 
    m^f_{t,\lambda}=0,\quad p_{t,\lambda}^f=\begin{pmatrix}
    0 & 0 \\ 0 & \id_{E_{\lambda}^\downarrow}
    \end{pmatrix},\quad
    q_{t,\lambda}^f=\begin{pmatrix}
    0 & 0 \\ 0 & \id_{E_{\lambda}^{\uparrow}}
    \end{pmatrix}.
\]
We shall show that for sufficiently large~\( t \) and sufficiently small~\(
\lambda\), then~\( p^e_{t,\lambda} \) and~\( p^f_{t,\lambda} \) are
quasi-projections in~\( \Mat_2(\Cpt_B(E)^+) \) that represents the same
class in~\( \K_0(\Cpt_B(E)^+) \). To this end, define
\[ 
    \Theta_t\defeq \begin{pmatrix}
    c_t & s_t \\ -s_t & c_t
    \end{pmatrix}\in\Mat_2(\End_B(E)).
\] 
Then~\( \Theta_t \) is a unitary multiplier of~\( \Mat_2(\Cpt_B(E)^+) \)
which commmutes with~\( D\oplus D \). Thus under any spectral decomposition
\( (E_\lambda^\downarrow,E_{\lambda}^\uparrow) \) for~\( D\oplus D \), we
have~\( \Theta_t \) is diagonal;~\( p^{\Theta}_{t,\lambda} \) and~\(
q^{\Theta}_{t,\lambda} \) are unitary operators on~\(
E_{\lambda}^\downarrow \) and~\( E_{\lambda}^{\uparrow} \), respectively.
Moreover, since~\( f_t \) also commutes with~\( D\oplus D \), we have  
\[ 
    q_{t,\lambda}^{\Theta}q_{t,\lambda}^fq_{t,\lambda}^\Theta{}^*=q_{t,\lambda}^{\Theta
    f\Theta^*},
\]
which is the truncation of the following operator onto~\(
E_{\lambda}^{\uparrow} \): 
\[ 
    \Theta_tf_t\Theta_t^*=\begin{pmatrix}
    c_t & s_t \\ -s_t & c_t
    \end{pmatrix}
    \begin{pmatrix}
    0 & 0 \\ 0 & \id
    \end{pmatrix}
    \begin{pmatrix}
    c_t & -s_t \\ s_t & c_t
    \end{pmatrix}
    =\begin{pmatrix}
    s_t^2 & c_ts_t \\ c_ts_t & c_t^2
    \end{pmatrix}.
\]

\begin{lemma}\label{lem:spectral_decomposition}
For each~\( t\in[1,\infty) \) and~\( \delta>0 \), 
there exists~\( \lambda>0 \) that depends on~\( t \) and~\( \delta \),
such that the following holds\textup{:}

If 
\( (E_{\lambda}^{\downarrow},E_{\lambda}^{\uparrow}) \) is a spectral
decomposition of~\( \widehat{E} \) for~\( D\oplus D \)
with spectral threshold~\( \lambda \). Then
\[
    \norm{q_{t,\lambda}^e-q^{\Theta}_{t,\lambda}q_{t,\lambda}^fq^{\Theta}_{t,\lambda}{}^*}\leq\delta.
\]
\end{lemma}

\begin{proof}
We compute
\begin{multline*}
\norm*{e_t-\Theta_tf_t\Theta_t^*}\leq\norm*{\begin{pmatrix}
    s_t^2 & \sqrt{c_ts_t}v\sqrt{c_ts_t} \\
    \sqrt{c_ts_t}v^*\sqrt{c_ts_t} & c_t^2
\end{pmatrix}-\begin{pmatrix}
    s_t^2 & \sqrt{c_ts_t}\sqrt{c_ts_t} \\
    \sqrt{c_ts_t}\sqrt{c_ts_t} & c_t^2
\end{pmatrix}} \\
=\norm*{\begin{pmatrix}
    0 & \sqrt{c_ts_t}(v-1)\sqrt{c_ts_t} \\
    \sqrt{c_ts_t}(v^*-1)\sqrt{c_ts_t} & 0
\end{pmatrix}
}
=\norm*{\sqrt{c_ts_t}(v-1)\sqrt{c_ts_t}}
\leq 2\norm*{c_ts_t}.
\end{multline*}
The functions~\( s,c\colon \R\to[0,1] \) are continuous and satisfy
\[
\lim_{x\to+\infty}s(x)=1,\quad\lim_{x\to+\infty}c(x)=0.
\]
So both of them extend to continuous functions~\( [-\infty,+\infty]\to[0,1]
\). The function~\( c(x)s(x) \) is positive and 
tends to~\( 0 \) as~\( x\to\infty \), that is, fix
\( t\in[1,+\infty) \) and~\( \delta>0 \), 
there exists~\( \lambda>0 \) such that~\(
0<c(x)s(x)<\frac{1}{2}\delta \) for all~\( x>t^{-1}\lambda \). Then
for all~\( \xi\in E_{\lambda}^{\uparrow} \), the operator inequality
\[ 
    \braket{c_ts_t\xi,c_ts_t\xi}
    \leq\frac{1}{2}\delta\cdot\braket{\xi,\xi}
\]
holds. This proves the claim.
\end{proof}

Providing the existence of a spectral decomposition
\( (E_\lambda^{\downarrow},E_\lambda^{\uparrow}) \),
the lemmas above allows us to obtain a pair of quasi-projections that
are supported on the lower spectral submodule~\( E_{\lambda}^{\downarrow}
\). 

\begin{theorem}\label{thm:index_pairing_submodule}
Let~\( A \) be a unital \Cst-algebra. Let~\( (\mathcal{A},E,D) \) be an
odd unbounded Kasparov~\( A \)-\( B \)--module and
\( v\in \mathcal{A} \) be a unitary. Write~\( [v]\in\K_1(A) \) and~\(
[D]\in\KK_1(A,B) \) for their representing classes.
Assume that the followings hold\textup{:} 
\begin{enumerate}
\item There exist~\( 0<\delta<\frac{1}{400} \) 
and a spectral decomposition~\(
(E_{\lambda}^{\downarrow},E_{\lambda}^{\uparrow}) \) of~\( \widehat{E} \) 
with spectral threshold~\( \lambda \) such that
\[ 
\norm*{q_{t,\lambda}^e-q_{t,\lambda}^{\Theta}q_{t,\lambda}^fq_{t,\lambda}^{\Theta}{}^*}\leq\delta.
\]
\item There exists~\( R>0 \) such that \eqref{eq:assmption_constants_R}
holds. 
\end{enumerate}
Then for any~\( 0<\varepsilon<\frac{1}{400}-\delta \) and any~\(
t\in[1,\infty) \) that satisfy  
\[  
    t>2\varepsilon^{-1}R\norm*{[D,v]},
\]
the \K-theoretic index pairing~\( \braket{[v],[D]}\in\K_0(\Cpt_B(E)) \) 
is represented by the formal difference of \emph{quasi-projections}
\[ 
    [p^e_{t,\lambda}]-[p^f_{t,\lambda}],
\]
where both~\( p^e_{t,\lambda} \) and~\( p^f_{t,\lambda} \) are
quasi-projections supported on the lower 
spectral submodule~\( E_{\lambda}^{\downarrow} \) as in
\eqref{eq:spectral_decomposition_et_ft}.
\end{theorem}

\begin{proof}
Since~\( \Theta_t \) is unitary and commutes with~\( D\oplus D \), 
it follows that~\( q^{\Theta}_{t,\lambda} \) is also a unitary operator on 
\( E_{\lambda}^{\uparrow} \)  for any
spectral decomposition~\( (E_{\lambda}^\downarrow,E_{\lambda}^\uparrow) \).
Thus the operator
\( q^{\Theta}_{t,\lambda}q^f_{t,\lambda}q^{\Theta}_{t,\lambda}{}^* \)
is a projection on~\( E_{\lambda}^{\uparrow} \),
which is unitarily equivalent to~\( q^f_{t,\lambda} \). 
Since 
\[
    \norm*{q^e_{t,\lambda}-q^{\Theta}_{t,\lambda}q^f_{t,\lambda}q^{\Theta}_{t,\lambda}{}^*}\leq\delta<\frac{1}{400}-\varepsilon,
\]
it follows from \Cref{thm:close_quasi-idempotents_homotopic} that~\(
q^e_{t,\lambda} \) is a~\( 4\delta \)-quasi-projection that is homotopic 
to~\( q^{\Theta}_{t,\lambda}q^f_{t,\lambda}q^{\Theta}_{t,\lambda}{}^* \),
thus
\[ 
    [q^e_{t,\lambda}]=[q^\Theta_{t,\lambda}q^f_{t,\lambda}
    q^{\Theta}_{t,\lambda}{}^*]
    =[q^f_{t,\lambda}]\in\K_0(\Cpt_B(E)^+).
\]

Also \Cref{lem:reduction_to_diagonal} implies that~\( e_t \) is
homotopic to~\( p^e_{t,\lambda}\oplus q^e_{t,\lambda} \) via a straightline
homotopy
\[ 
    [0,1]\to\End_B(\widehat{E}),\quad s\mapsto\begin{pmatrix}
    p^e_{t,\lambda} & s\cdot m^e_{t,\lambda}{}^* \\ 
    s\cdot m^e_{t,\lambda} & q^f_{t,\lambda}
    \end{pmatrix}.
\]
In particular, since~\( e_t\in\Mat_2(\Cpt_B(E)^+) \), it follows that the
operator
\[ 
    \begin{pmatrix}
    0 & m_{t,\lambda}^e{}^* \\
    m_{t,\lambda}^e{}^* & 0
    \end{pmatrix}  
\]
is compact. This shows that this straightline homotopy actually
lands in~\( \Mat_2(\Cpt_B(E)^+) \). Therefore, we have
\[ 
[e_t]=\left[\begin{pmatrix}
p^e_{t,\lambda} & 0 \\ 0 & q^e_{t,\lambda}
\end{pmatrix}\right]=[p^e_{t,\lambda}]+[q^e_{t,\lambda}]\in\K_0(\Cpt_B(E)^+),
\]
and
\[ 
    [q^e_{t,\lambda}]=[q^f_{t,\lambda}]\in\K_0(\Cpt_B(E)^+).
\]
On the other hand, the operator~\( f_t \) is diagonal under the
spectral decomposition~\(
\widehat{E}\simeq E_{\lambda}^{\downarrow}\oplus E_{\lambda}^{\uparrow}
\). So  
\[
    [f_t]=[p^f_{t,\lambda}]+[q^f_{t,\lambda}].
\]
Therefore, we have
\[ 
    [e_t]-[f_t]=[p^e_{t,\lambda}]+[q^e_{t,\lambda}]-[p^f_{t,\lambda}]-[q^f_{t,\lambda}]=[p^e_{t,\lambda}]-[p^f_{t,\lambda}].\qedhere
\]
\end{proof}

Note that the choice of the lower spectral threshold~\( \lambda \) 
depends on~\( \delta \) and~\( t \).  
As~\( t \) increases, then the threshold
\( \lambda \) will typically also increase.
Nevertheless, for each \emph{fixed}~\( t\in
[1,\infty) \), legitimated by the condition
\[ 
    t>2\varepsilon^{-1}R\norm*{[D,v]}
\]
as in \Cref{thm:index_pairing_quasi-projections}, we are able to produce a
finite spectral threshold~\( \lambda \). Assume that there exists a
spectral decomposition~\( (E_\lambda^\downarrow,E_\lambda^\uparrow) \) with
spectral threshold~\( \lambda \), then we are able to ``spectrally''
truncate~\( [e_t]-[f_t] \) to the lower spectral submodule~\(
E_{\lambda}^{\downarrow} \), which represents the same K-theoretic index
pairing. 

\subsection{Emergence of the spectral localiser}
\label{sec:spectral_localiser_spectral_triple}
The precise formula for~\( p^e_{t,\lambda} \) and~\( p^f_{t,\lambda} \), as
well as their relevant constants~\( R>0 \) and~\( \lambda>0 \), are subject
to the specific choices of functions~\( c \) and~\( s \). The choices that
we will be employing to connect to the spectral localiser are given as
follows.

\begin{theorem}\label{thm:choice_cs}
Let
\begin{equation}\label{eq:choice_cs}
    c(x)\defeq\sqrt{\frac{1}{2}-\frac{1}{2}x(1+x^2)^{-\sfrac{1}{2}}},\quad 
    s(x)\defeq\sqrt{\frac{1}{2}+\frac{1}{2}x(1+x^2)^{-\sfrac{1}{2}}},
\end{equation}
\textup{(}see \Cref{fig:cs}\textup{)}. Then\textup{:}
\begin{enumerate}
\item The constant~\( R \) in \eqref{eq:assmption_constants_R} can be
chosen to be~\( 2 \).
\item The spectral threshold~\( \lambda \) as in
\Cref{lem:spectral_decomposition} can be chosen to be~\( t\delta^{-1}
\). 
\end{enumerate}
\end{theorem}

\begin{proof}
What (1) claims is the following: the following
estimation holds for all~\( t\in[1,\infty) \): 
\[
\left.\begin{array}{r}
\norm*{[c_t,v]} \\
\norm*{[s_t,v]} \\
\norm*{[\sqrt{c_ts_t},v]}
\end{array}\right\}\leq 2t^{-1}\norm*{[D,v]}.
\]
We provide the proof of (1) in \Cref{app:commutator_estimates}, based on
an integral formula as well as trigonometric substitutions. To see
(2), we note that
\[ 
    c(x)s(x)=\sqrt{\frac{1}{4}-\frac{1}{4}x^2(1+x^2)^{-1}}=\frac{1}{2}\frac{1}{\sqrt{x^2+1}}.
\]
So~\( c_ts_t=\frac{1}{2}(1+t^{-2}D^2)^{-\sfrac{1}{2}} \). Therefore, for
each~\( \delta>0 \), choose~\( \lambda>t\delta^{-1} \), it follows from
\Cref{lem:spectral_decomposition} that
\[
\norm*{q^e_{t,\lambda}-q^{\Theta}_{t,\lambda}q^f_{t,\lambda}q^{\Theta}_{t,\lambda}{}^*}\leq
\frac{1}{2}\frac{1}{\sqrt{1+t^{-2}\lambda^2}}
\leq\frac{1}{2t^{-1}\lambda}\leq
\frac{1}{2}\delta.\qedhere \]
\end{proof}

\begin{figure}[h!]
\centering
\includegraphics[width=0.8\textwidth]{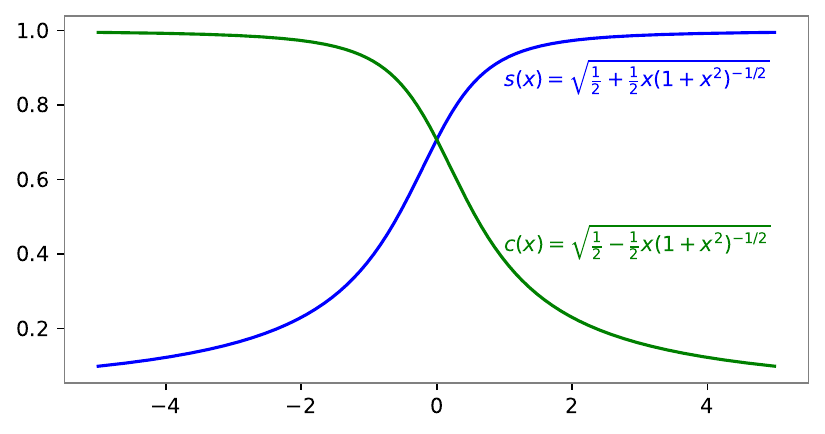}
\caption{The functions~\( c(x) \) and~\( s(x) \) in
\eqref{eq:choice_cs}.}
\label{fig:cs}
\end{figure}

Now we connect our results to the (odd) spectral localisers of Loring
and Schulz-Baldes.
In this case of interest, we let~\( B=\C \). Then an 
unbounded Kasparov~\( A \)-\( B \)--module is the same thing
as an odd special triple 
\( (\mathcal{A},\mathcal{H},D) \) over~\( A \).
As explained in \Cref{ex:spectral_decomposition_Hilbert_space}, any Borel
subset~\( I\subseteq(-\infty,-\lambda)\cup(\lambda,+\infty) \) gives a
spectral decomposition~\( (\mathcal{H}_{\R\setminus I},\mathcal{H}_I) \) 
with spectral threshold~\( \lambda \). Following the notations in
\cites{Loring-SBaldes:Odd_spectral_localiser,Loring-SBaldes:Odd_spectral_localiser_sf},
we set
\begin{gather*} 
\mathcal{H}_{\lambda}\defeq
\chi_{[-\lambda,\lambda]}(D\oplus D)\widehat{\mathcal{H}},\\
\mathcal{H}_{\lambda^\mathrm{c}}\defeq
\chi_{(-\infty,-\lambda)\cup(\lambda,\infty)}(D\oplus
D)\widehat{H}.
\end{gather*}
Then~\( (\mathcal{H}_{\lambda},\mathcal{H}_{\lambda^\mathrm{c}}) \) is a
spectral decomposition of~\( \widehat{\mathcal{H}} \) with spectral
threshold~\( \lambda \).  

The construction in the previous section gives 
a pair of quasi-projections~\( p^e_{t,\lambda} \)
and~\( p^f_{t,\lambda} \) on the spectral space~\(
\mathcal{H}_{\lambda} \). In particular,~\( \mathcal{H}_\lambda \) is
\emph{finite-dimensional} as~\( D \) has discrete spectrum.  
The formal difference
\[ 
    [p^e_t]-[p^f_t]
\]
defines a class in~\( \K_0(\Cpt)\simeq\Z \), which 
coincides with the index pairing~\( \braket{[v],[D]} \). 
Since both~\( p^e_{t,\lambda}
\) and~\( p^f_{t,\lambda} \) have finite rank, it follows that the class  
represented by the formal difference is
\[ 
    \rank(\kappa_0(p^e_{t,\lambda}))-\rank(\kappa_0(p^f_{t,\lambda})).
\]

Now we relate the
constructions in the previous chapter to the spectral localiser.
We begin with the following simple observation.

\begin{lemma}\label{lem:signature_quasi-projection}
Let~\( p \) be a quasi-projection on a finite-dimensional Hilbert space~\(
\mathcal{K}\), then
\begin{align*} 
    \sig(2p-\id_\mathcal{K})&=\rank(\kappa_0(p))-\rank(\kappa_0(\id_\mathcal{K}-p))\\ &=2\rank(\kappa_0(p))-\dim\mathcal{K}.
\end{align*}
\end{lemma}

\begin{proof}
Recall that the signature of a finite-dimensional, self-adjoint matrix~\( L
\) is given by
\[ 
    \sig(L)=\#(\text{positive eigenvalues of~\( L \)})-\#(\text{negative
    eigenvalues of~\( L \)}).
\]
Thus the signature of~\( L \) is the trace of the sign of~\( L \):  
\[ 
    \sig(L)=\tr(\operatorname{sign}(L)),\quad
    \operatorname{sign}(x)\defeq\begin{cases}
    1 & x>0; \\
    0 & x=0; \\
    -1 & x<0.
    \end{cases}
\]
Notice that
\[ 
    \operatorname{sign}(2x-1)=\begin{cases}
    1 & x>\sfrac{1}{2}; \\
    0 & x=\sfrac{1}{2}; \\
    -1 & x<\sfrac{1}{2}.
    \end{cases}    
\]
Since~\( p \) is a quasi-projection on~\( \mathcal{K} \), 
we have~\( \sfrac{1}{2}\notin\spc(p)
\) and~\( \sfrac{1}{2}\notin\spc(\id_{\mathcal{K}}-p) \), 
hence~\( \operatorname{sign}(2x-1) \) coincides with
\[
    2\kappa_0(x)-1=\kappa_0(x)-(1-\kappa_0(x))=\kappa_0(x)-\kappa_0(1-x)
\]
on the spectrum of~\( p \). Now the claim follows as
\begin{multline*} 
\sig(2p-\id_\mathcal{K})=\tr(\operatorname{sign}(2p-\id_\mathcal{K}))=\tr(2\kappa_0(p)-\id_\mathcal{K})\\
=2\tr\kappa_0(p)-\tr\id_{\mathcal{K}}=2\rank\kappa_0(p)-\dim\mathcal{K}.\qedhere
\end{multline*}
\end{proof}

\begin{remark}
We note that this is also the reason to replace quasi-idempotents~\(
\check{e}_t \) by quasi-projections~\( e_t \), since the signature is
only defined for self-adjoint matrices.  
\end{remark}

Now we recall the definition of (odd) spectral localisers of Loring and
Schulz-Baldes. Let~\( \kappa>0 \), define the following unbounded,
self-adjoint operator
\[ 
    L_{\kappa}\defeq \begin{pmatrix}
    \kappa D & v \\ v^* & -\kappa D
    \end{pmatrix}
\]
and set~\( L_{\kappa,\lambda} \) to be the truncation of~\( L_\kappa \) to~\(
\mathcal{H}_{\lambda} \), which is a self-adjoint matrix of dimension~\(
\dim\mathcal{H}_{\lambda} \). 
Then it was proven in
\cites{Loring-SBaldes:Odd_spectral_localiser,Loring-SBaldes:Odd_spectral_localiser_sf}
that for sufficiently small~\( \kappa<0 \) and sufficiently large~\( \lambda>0
\), it follows that the index pairing~\( \braket{[v],[D]} \) equals one
half of the signature of~\( L_{\kappa,\lambda} \). The precise conditions for
the tuning paramter~\( \kappa>0 \) and the spectral threshold~\( \lambda>0 \)
can be found in
\cite{Loring-SBaldes:Odd_spectral_localiser_sf}*{Page 2}.

We shall prove that the invertible matrix~\(
\mathcal{L}^e_{t^{-1},\lambda}\defeq 2p^e_{t,\lambda}-\id \) 
is congruent to the spectral localiser~\( L_{t^{-1},\lambda} \).
Controls of the tuning parameter~\( \kappa=t^{-1} \) and the spectral
threshold~\( \lambda \) are given by \Cref{thm:index_pairing_submodule}
(for the general setup) and 
\Cref{thm:choice_cs} (for the specific choice of functions~\( c(x) \) and
\( s(x) \)).

\begin{theorem}\label{thm:spectral_localiser}
For every~\( t,\lambda>0 \), there is an equality
\[ 
    \frac{1}{2}\sig(L_{t^{-1},\lambda})=
    \rank(\kappa_0(p^e_{t,\lambda}))-\rank(\kappa_0(p^f_{t,\lambda})).
\]
As a consequence, let~\( \varepsilon,\delta>0 \) satisfy~\(
\varepsilon+\delta<\frac{1}{400} \). For any choice~\( t,\lambda \)
satisfying~\(
t>4\varepsilon^{-1}\norm*{[D,v]} \) and~\( \lambda>t\delta^{-1}
\), the half-signature 
\( \frac{1}{2}\sig(L_{t^{-1},\lambda}) \) coincides with the index
pairing~\( \braket{[v],[D]} \).   
\end{theorem}

\begin{proof}
It follows from \Cref{thm:index_pairing_submodule,thm:choice_cs} 
that under the conditions for~\( \varepsilon,\delta,t
\) and~\( \lambda \), the index pairing~\( \braket{[v],[D]} \) is
given by
\[
    \rank(\kappa_0(p^e_{t,\lambda}))-\rank(\kappa_0(p^f_{t,\lambda}))\in\Z\simeq\K_0(\Cpt).
\]
By the preceding lemma, we have
\begin{align*} 
\rank(\kappa_0(p^e_{t,\lambda}))-\rank(\kappa_0(p^f_{t,\lambda}))&=
\frac{1}{2}\left(\sig(2p^e_{t,\lambda}-\id)+\dim\mathcal{H}_{\lambda}\right)-
\frac{1}{2}\left(\sig(2p^f_{t,\lambda}-\id)+\dim\mathcal{H}_{\lambda}\right)\\
&=\frac{1}{2}\sig(2p^e_{t,\lambda}-\id)-\frac{1}{2}\sig(2p^f_{t,\lambda}-\id).
\end{align*}
The operator~\( p_{t,\lambda}^f \) is given by 
\(\begin{pmatrix}
    0 & 0 \\ 0 & \id
\end{pmatrix}\). Hence
\[
    \sig(2p_{t,\lambda}^f-\id)=\sig\begin{pmatrix}
    -\id & 0 \\ 0 & \id
    \end{pmatrix}=0.
\]

Now let
\[
    \mathcal{L}^e_{t^{-1},\lambda}\defeq
    2p^e_{t,\lambda}-\id.
\]
Then~\( \mathcal{L}^e_{t^{-1},\lambda} \) is
the truncation of the following operator onto~\( \mathcal{H}_{\lambda}
\):
\[ 
    2\begin{pmatrix}
    s_t^2 & \sqrt{c_ts_t}v\sqrt{c_ts_t} \\
    \sqrt{c_ts_t}v^*\sqrt{c_ts_t} & c_t^2
    \end{pmatrix}-\begin{pmatrix}
    \id & 0 \\ 0 & \id
    \end{pmatrix}=
    \begin{pmatrix}
    2s_t^2-\id & \sqrt{2c_ts_t}v\sqrt{2c_ts_t} \\
    \sqrt{2c_ts_t}v^*\sqrt{2c_ts_t} & 2c_t^2-\id
    \end{pmatrix}.
\]
Abbreviate~\( D_t\defeq t^{-1}D \). We compute that
\begin{alignat*}{3}
2s_t^2-\id&=\id+D_t(1+D_t^2)^{-\sfrac{1}{2}}-\id&&=D_t(1+D_t^2)^{-\sfrac{1}{2}};\\
2c_t^2-\id&=\id-D_t(1+D_t^2)^{-\sfrac{1}{2}}-\id&&=-D_t(1+D_t^2)^{-\sfrac{1}{2}};\\
2c_ts_t&=2\sqrt{\frac{1}{4}-\frac{1}{4}D_t^2(1+D_t^2)^{-1}}&&=(1+D_t^2)^{-\sfrac{1}{2}}.
\end{alignat*}
Thus  
\[
    \mathcal{L}^e_{t^{-1},\lambda}=\begin{pmatrix}
    D_t(1+D_t^2)^{-\sfrac{1}{2}} &
    (1+D_t^2)^{-\sfrac{1}{4}}v(1+D_t^2)^{-\sfrac{1}{4}} \\
    (1+D_t^2)^{-\sfrac{1}{4}}v^*(1+D_t^2)^{-\sfrac{1}{4}} &
    -D_t(1+D_t^2)^{-\sfrac{1}{2}} &
    \end{pmatrix}
\]
where the right-hand side is considered as an operator on~\(
\mathcal{H}_\lambda \) through truncation. The operator~\(
(1+D_t^2)^{-\sfrac{1}{4}} \) is a positive, invertible operator on~\(
\mathcal{H} \) which commutes with the spectral projections of~\( D \), 
thus its truncation onto~\( \mathcal{H}_\lambda \) is also
positive and invertible. So we have
\[ 
    \mathcal{L}^e_{t^{-1},\lambda}=\begin{pmatrix}
    (1+D_t^2)^{-\sfrac{1}{4}} & 0 \\ 0 & (1+D_t^2)^{-\sfrac{1}{4}} 
    \end{pmatrix}
    \begin{pmatrix}
    D_t & v \\
    v^* & -D_t
    \end{pmatrix}
    \begin{pmatrix}
    (1+D_t^2)^{-\sfrac{1}{4}} & 0 \\ 0 & (1+D_t^2)^{-\sfrac{1}{4}} 
    \end{pmatrix}.
\]
Thus~\( \sig(\mathcal{L}^e_{t^{-1},\lambda}) \)
coincides with the signature of the truncation of
\[ 
    \begin{pmatrix}
    D_t & v \\ v^* & -D_t
    \end{pmatrix}=\begin{pmatrix}
    t^{-1}D & v \\ v^* & -t^{-1}D
    \end{pmatrix}.
\]
Then we conclude that
\begin{multline*} 
    \braket{[v],[D]}=\rank(\kappa_0(p^e_{t,\lambda}))-\rank(\kappa_0(p^f_{t,\lambda}))=\frac{1}{2}\sig(2p^e_{t,\lambda}-1)-\frac{1}{2}\sig(2p^f_{t,\lambda}-1)\\
    =\frac{1}{2}\sig\mathcal{L}^e_{t^{-1},\lambda}=\frac{1}{2}\sig L_{t^{-1},\lambda}.\qedhere
\end{multline*}
\end{proof}

\subsection{Numerical index pairings and their truncations}
In this last section, we shall explain how a large class of \emph{numerical}
index pairing can be computed using the spectral truncation method of
\Cref{thm:index_pairing_submodule}. In particular, a \emph{semi-finite}
version of the spectral localiser 
defined by Schulz-Baldes and Stoiber
\cite{SBaldes-Stoiber:Semifinite_spectral_localizer}
arises in this way.

The setting that we shall be working in is as follows. Let~\( A \) and~\( B
\) be unital \Cst-algebras,~\( \tau \) a \emph{finite}, 
\emph{faithful} trace on~\( B \).  Then~\( \tau
\) extends to~\( \Mat_n(B) \) for each~\( n\in\N \) by (still denoted by~\(
\tau\):)
\[
    \tau\colon\Mat_{n}(B)\to\C,\quad (b_{ij})_{i,j=1}^{n}\mapsto\sum_{i=1}^n
    \tau(b_{ii}),
\]
and hence induces a group homomorphism
(cf.~\cite{Willett-Yu:Higher_index_theory}*{Remark 2.1.6}):
\[ 
    \tau_*\colon \K_0(B)\to\R.
\]

As before, we let~\( (\mathcal{A},E,D) \) be an odd unbounded Kasparov~\( A
\)-\( B \)--module and~\( v\in\mathcal{A} \) be a unitary. Let~\(
[v]\times[D]\in\K_0(B) \) be their Kasparov product as in
\Cref{eq:K-theoretic_index_pairing}. Their~\( \tau \)-(\emph{numerical}) 
\emph{index pairing} is defined as the real number
\[ 
    \tau_*([v]\times[D])\in\R.
\]

We wish to compute the~\( \tau \)-index pairing using the quasi-projection
representative~\( [e_t]-[f_t]\in\K_0(\Cpt_B(E)) \) as in
\Cref{thm:index_pairing_quasi-projections}. Next, we shall apply 
the spectral truncation method as in 
\Cref{thm:index_pairing_submodule} to replace~\(
[e_t]-[f_t] \) by their truncations onto a spectral submodule. The first
step is established by tranferring the finite, faithful trace~\(
\tau\colon B\to\C \) to a semi-finite trace on~\( \Cpt_B(E) \) and
then applying \Cref{thm:commutative_diagram_trace} below.

A finite trace~\( \tau\colon B\to\C \) is completely positive. Hence, it
yields, via 
the KSGNS construction (cf.~\cite{Lance:Hilbert_Cst-modules}*{Chapter 5}),
a Hilbert space
\[
    \mathcal{H}^\tau\defeq E\otimes_BL^2(B,\tau),
\]
obtained from the interior tensor product of the Hilbert~\( B \)-module~\(
E\) with the GNS-Hilbert space~\( L^2(B,\tau) \) constructed from~\(
\tau\colon B\to\C\). 

The Hilbert space~\( \mathcal{H}^\tau \) carries a representation
\[ 
    \varrho^\tau\colon \End_B^*(E)\to\Bdd(\mathcal{H}^\tau),\quad
    \varrho^\tau(T)(\xi\otimes b)\defeq T\xi\otimes b,
\]
which is isometric because~\( \tau \) is faithful. Thus
we may identify~\( \End_B^*(E) \), as well as its \Cst-subalgebra~\(
\Cpt_B(E) \), with their images under~\( \varrho^\tau \).  

Now we transfer the finite trace~\( \tau\colon B\to\C \) to a
\emph{semi-finite} trace~\( \widehat{{\tau}} \) on~\( \Cpt_B(E) \). The
following presentation comes from Laca and Neshveyev, whereas they 
also attributed a more general result to the work of 
Combes and Zettl 
\cite{Combes-Zettl:Order_structures}*{Proposition 2.1 and 2.2}, in which 
\( \tau \) is merely assumed to be densely defined. 

\begin{lemma}[\cite{Laca-Neshveyev:KMS_states_quasi-free}*{Theorem
1}]\label{lem:transfer_trace}
Let~\( B \) be a \Cst-algebra with a \emph{finite} trace~\( \tau\colon
B\to\C \). Let~\( \mathcal{E} \) be a Hilbert~\( B \)-module. 
Denote by~\( \ket{\xi_1}\bra{\xi_2} \) the rank-one
operator
\( \ket{\xi_1}\bra{\xi_2}\xi_3\defeq \xi_1\braket{\xi_2,\xi_3} \).
Define
\[ 
    \widehat{\tau}\colon\End^*_{B}(\mathcal{E})_+\to[0,+\infty],\quad 
    \widehat{\tau}(T)\defeq \sup_I\sum_{\xi\in I}\tau(\braket{\xi,T\xi}),
\]
where~\( I \) ranges over all finite subsets of~\( \mathcal{E} \) such that
\( \sum_{\xi\in I}\ket{\xi}\bra{\xi}\leq \id_\mathcal{E} \). 
Extend~\( \widehat{\tau} \) by linearality. Then the following 
holds\textup{:}
\begin{enumerate}
\item \( \widehat{\tau} \) is a semi-finite trace on~\(
\End_{B}^*(\mathcal{E}) \). 
\item 
\( \widehat{\tau}(\ket{\xi_1}\bra{\xi_2})=\tau(\braket{\xi_2,\xi_1}) \) 
and~\( \widehat{\tau}(\ket{\xi_1}\bra{\xi_2})<+\infty \). 
As a consequence, every
element in~\( \Fin_B^*(\mathcal{E}) \) is~\( \widehat{\tau} \)-finite.  
\end{enumerate}
\end{lemma}

The semi-finite trace~\( \widehat{\tau} \) on~\( \End_B^*(E) \) extends to
the enveloping von Neumann algebra~\(
\mathcal{N} \) of~\( \varrho^\tau(\End_B^*(E)) \).
Denote by~\( \Fin_{\widehat{\tau}} \) the \st-subalgebra in~\( \mathcal{N}
\) generated by all~\( \tau \)-finite projections, 
and~\( \Cpt_{\widehat{\tau}} \) be the
norm-closed ideal generated by~\( \Fin_{\widehat{\tau}} \).    
It is shown in \cite{Kaad-Nest-Rennie:KK-theory_spectral_flow}*{Lemma
6.3} that~\( \Fin_{\widehat{\tau}} \) is
holomorphically closed in~\( \Cpt_{\widehat{\tau}} \). This gives a
well-defined group homomorphism 
(cf.~\cite{Kaad-Nest-Rennie:KK-theory_spectral_flow}*{Theorem 6.4}):
\[ 
    \widehat{\tau}_*\colon \K_0(\Cpt_{\widehat{\tau}})\to\R,\quad
    \widehat{\tau}_*([p]-[q])\defeq 
    \widehat{\tau}\big(p-s(p)\big)-
    \widehat{\tau}\big(q-s(q)\big),
\]
where~\( p,q\in\Mat_n(\Fin_{\widehat{\tau}}^+) \) are projections
such that~\( [s(p)]=[s(q)] \) in~\( \K_0(\C) \), and~\( s\colon
\Mat_n(\Fin_{\widehat{\tau}}^+)\twoheadrightarrow\Mat_n(\C) \)
is the projection onto the scalar matrix part.

\begin{theorem}\label{thm:commutative_diagram_trace}
The \st-homomorphism~\( \varrho^{\tau}\colon
\Cpt_B(E)\to\Bdd(\mathcal{H}^\tau) \)
maps~\( \Cpt_B(E) \) into~\( \Cpt_{\widehat{\tau}} \). Its induced map
\[ 
    \varrho^{\tau}_*\colon \K_0(\Cpt_B(E))\to\K_0(\Cpt_{\widehat{\tau}})
\]
makes the the following diagram commute:
\[ \begin{tikzcd}[row sep=large, column sep=large]
\K_1(A) \arrow[r, "\times{[D]}"] \arrow[rd, "\Phi^D_*"'] & \K_0(B)
\arrow[r, "\tau_*"] & \R \\
& \K_0(\Cpt_B(E))  \arrow[u, "{[E]}"]  \arrow[r,"\varrho^{\tau}_*"] &
\K_0(\Cpt_{\widehat{\tau}}) \arrow[u, "\widehat{\tau}_*"'].
\end{tikzcd} \]
\end{theorem}

We provide a detailed proof in \Cref{app:proof_commutativity}.
It follows immediately from the preceding theorem that:

\begin{corollary}
The~\( \tau \)-index pairing~\( \tau_*([v]\times[D])\in\R \) is
represented by
\[ 
    \widehat{\tau}_*\circ\varrho^\tau_*([e_t]-[f_t])=
    \widehat{\tau}_*([\varrho^\tau(e_t)]-[\varrho^\tau(f_t)]),
\]
where
\[ 
    e_t=\begin{pmatrix}
    s_t^2 & \sqrt{c_ts_t} v\sqrt{c_ts_t} \\
    \sqrt{c_ts_t} v^*\sqrt{c_ts_t} &  c_t^2
    \end{pmatrix},\quad f_t=\begin{pmatrix}
    0 & 0 \\ 0 & \id
    \end{pmatrix}
\]
are \textup{(}quasi-\textup{)}projections in~\( \Mat_2(\Cpt_B(E)) \), and
\( \varrho^\tau(e_t),\varrho^\tau(f_t)\in\Mat_2(\Cpt_{\widehat{\tau}^+}) \)
are their images under the injective \st-homomorphism~\(
\varrho^\tau\otimes\mathbf{1}_{2\times 2}\colon
\Mat_2(\Cpt_B(E)^+)\to\Mat_2(\Cpt_{\widehat{\tau}}^+) \). 
\end{corollary}

Now we apply spectral truncation to the quasi-projection
representative
\[
    [\varrho^\tau(e_t)]-[\varrho^\tau(f_t)]\in\K_0(\Cpt_{\widehat{\tau}}).
\]
where both~\( \varrho^\tau(e_t) \) and~\( \varrho^\tau(f_t) \) are
represented faithfully on the Hilbert
\emph{space}~\( \widehat{\mathcal{H}}^\tau\defeq \mathcal{H}^\tau\oplus
\mathcal{H}^\tau \). Hence, as explained in
\Cref{ex:spectral_decomposition_Hilbert_space}, 
a spectral decomposition may be
constructed from the spectral projections of an unbounded, self-adjoint
operator thereon.

In the case of interest, we consider spectral projections of the
unbounded, self-adjoint operator
\( D^\tau\oplus D^\tau \) obtained from the unbounded, self-adjoint and
regular operator~\( D \) on~\( E \), localised at the representation~\(
\End_B^*(E)\to\Bdd(\mathcal{H}^\tau) \). That is, let~\( D\otimes\id \)
denote the interior tensor product operator  
\[ 
    D\otimes\id\colon
    \dom D\otimes^{\mathrm{alg}}_BL^2(E,\tau)\subseteq
    \mathcal{H}^\tau\to\mathcal{H}^\tau,\quad
    (D\otimes\id)(\xi\otimes b)\defeq D\xi\otimes b,
\]
defined on the algebraic tensor product~\(
\dom D\otimes^{\mathrm{alg}}_BL^2(E,\tau) \). Then its closure, denoted by
\( D^\tau \), is a self-adjoint operator on~\( \mathcal{H}^\tau \) 
(cf.~\cite{Lance:Hilbert_Cst-modules}*{Chapter 10}).

Let~\( \lambda>0 \), we set
\begin{gather*} 
\mathcal{H}^{\tau}_{\lambda}\defeq
\chi_{[-\lambda,\lambda]}(D^\tau\oplus D^\tau)\widehat{\mathcal{H}}^\tau,\\
\mathcal{H}^{\tau}_{\lambda^\mathrm{c}}\defeq
\chi_{(-\infty,-\lambda)\cup(\lambda,\infty)}(D^\tau\oplus
D^\tau)\widehat{\mathcal{H}}^\tau.
\end{gather*}
Then~\( (\mathcal{H}_\lambda^\tau,\mathcal{H}_{\lambda^\mathrm{c}}^\tau) \)
is a spectral decomposition for the unbounded, self-adjoint operator~\(
D^\tau\oplus D^\tau \) with spectral threshold~\( \lambda \).
Similar as in \eqref{eq:spectral_decomposition_et_ft}, we write
\begin{equation}\label{eq:spectral_decomposition_et_ft_semifinite}
    \varrho^{\tau}(e_t)=\begin{pmatrix} 
    p^e_{t,\lambda} & {m^e_{t,\lambda}}^* \\
    m^e_{t,\lambda} & q^e_{t,\lambda} 
    \end{pmatrix},\quad
    \varrho^{\tau}(f_t)=\begin{pmatrix} 
    p^f_{t,\lambda} & {m^f_{t,\lambda}}^* \\
    m^f_{t,\lambda} & q^f_{t,\lambda} 
    \end{pmatrix},\quad
    \varrho^{\tau}(\Theta_t)=\begin{pmatrix} 
    p^\Theta_{t,\lambda} & {m^\Theta_{t,\lambda}}^* \\
    m^\Theta_{t,\lambda} & q^\Theta_{t,\lambda} 
    \end{pmatrix},
\end{equation}
where
\begin{align*}
p^e_{t,\lambda},p^f_{t,\lambda},p^{\Theta}_{t,\lambda}&\colon
\mathcal{H}^{\tau}_{\lambda}\to\mathcal{H}^\tau_{\lambda}, \\
m^e_{t,\lambda},m^f_{t,\lambda},m^\Theta_{t,\lambda}&\colon
\mathcal{H}^\tau_{\lambda}\to\mathcal{H}^\tau_{\lambda^\mathrm{c}}, \\
q^e_{t,\lambda},q^f_{t,\lambda},q^\Theta_{t,\lambda}&\colon
\mathcal{H}^\tau_{\lambda^\mathrm{c}}\to\mathcal{H}^\tau_{\lambda^\mathrm{c}}.
\end{align*}

Replacing~\( [\varrho^{\tau}(e_t)]-[\varrho^{\tau}(f_t)] \) by their
spectral truncations~\( [p_{t,\lambda}^e]-[p_{t,\lambda}^f] \) requires an
analogue of \Cref{lem:reduction_to_diagonal}, which allows for reducing
\( \varrho^{\tau}(e_t) \) to its spectral diagonal. In order to do this,
one key step is showing that~\( \chi_K(D^\tau) \) is
\( \widehat{\tau} \)-finite for every compact set~\( K \). 
We prove this in \Cref{lem:resolvent_tau_compact}, based on the remarkable
\Cref{lem:compactly_supported_fin_rank}, which states that 
\( f(D)\in\Fin_B^*(E) \) for all~\( f\in\Cc(\R) \).   

\begin{theorem}\label{thm:semi-finite_index_pairing_subspace}
Let~\( A \) and~\( B \)  be unital \Cst-algebras,~\( \tau \) be
finite faithful trace on~\( B \). Let~\( (\mathcal{A},E,D) \) be an
odd unbounded Kasparov~\( A \)-\( B \)--module and
\( v\in \mathcal{A} \) be a unitary. Write~\( [v]\in\K_1(A) \) and~\(
[D]\in\KK_1(A,B) \) for their representing classes. Then\textup{:}
\begin{enumerate}
\item For every~\( 0<\delta<\frac{1}{400} \), 
there exists~\( \lambda>0 \), such that under the spectral
decomposition~\(
(\mathcal{H}_\lambda^{\tau},\mathcal{H}_{\lambda^\mathrm{c}}^{\tau}) \), it
follows that
\( \norm{q_{t,\lambda}^e-q_{t,\lambda}^\Theta
q_{t,\lambda}^fq_{t,\lambda}^{\Theta}{}^*}\leq\delta \)\textup{;}
\item Assume that there exists~\( R>0 \) such that 
\eqref{eq:assmption_constants_R}
holds. 
Then for any~\( 0<\varepsilon<\frac{1}{400}-\delta \) and any~\(
t\in[1,\infty) \) that satisfy  
\[  
    t>2\varepsilon^{-1}R\norm*{[D,v]},
\]
then
\[ 
[\varrho^\tau(e_t)]-[\varrho^\tau(f_t)]=[p^e_{t,\lambda}]-[p^f_{t,\lambda}]\in\K_0(\Cpt_{\widehat{\tau}}).
\]
\end{enumerate}
\end{theorem}

\begin{proof}
The proof follows in a similar way as
\Cref{thm:index_pairing_submodule}, except that we
must show that the homotopy
\[ 
    [0,1]\to\Bdd(\widehat{\mathcal{H}}^\tau),\quad
    s\mapsto\begin{pmatrix}
    p^e_{t,\lambda} & s\cdot m^e_{t,\lambda}{}^* \\ 
    s\cdot m^e_{t,\lambda} & q^f_{t,\lambda}
    \end{pmatrix}
\]
lies inside~\( \Mat_2(\Cpt_{\widehat{\tau}}^+) \). To this end, write~\(
P_\lambda^{\tau}\colon
\widehat{\mathcal{H}}^\tau\to\mathcal{H}_{\lambda}^\tau \) for the
orthogonal projection onto~\( \mathcal{H}_\lambda^\tau \),
then~\( P_\lambda^\tau=\chi_{[-\lambda,\lambda]}(D^\tau\oplus D^\tau) \).
Since~\(
[-\lambda,\lambda] \) is a compact subset of~\( \R \), it follows from
\Cref{lem:resolvent_tau_compact} that~\( P_\lambda^{\tau} \) is a~\(
\widehat{\tau} \)-finite projection.  Since~\(
\Mat_2(\Cpt_{\widehat{\tau}})\subseteq \Mat_2(\mathcal{N}) \) is an ideal,
we conclude that~\( m^e_{t,\lambda}=m^e_{t,\lambda}P_\lambda^\tau \)
belongs to~\( \Mat_2(\Cpt_{\widehat{\tau}}) \) as well. 
\end{proof}

Similar as in the spectral triple case of
\Cref{sec:spectral_localiser_spectral_triple}, 
define the following unbounded,
self-adjoint operator on~\( \mathcal{H}^\tau \): 
\[ 
    L^\tau_\kappa\defeq \begin{pmatrix}
    \kappa D^\tau & \varrho^\tau(v) \\
    \varrho^\tau(v^*) & -\kappa D^\tau
    \end{pmatrix}
\]
and let~\( L^\tau_{\kappa,\lambda} \) be its truncation onto~\(
\mathcal{H}^\tau_{\lambda} \). We shall show that the~\(
\tau \)-index pairing~\( \tau_*\braket{[v],[D]} \) can be obtained as the
half--\( \tau \)-signature (defined below) of~\( L^\tau_{\kappa,\lambda} \)
for sufficiently small tuning parameter~\( \kappa \) and sufficiently large
spectral threshold~\( \lambda \).  

\begin{definition}\label{def:tau_rank_sig}
Let~\( A \) be a \Cst-algebra, and~\( \tau \) be a densely defined trace on
\( A \).
\begin{enumerate}
\item Let~\( p\in A \) be a quasi-projection such that~\( \kappa_0(p) \) is
\( \tau \)-finite. Then the~\( \tau \)-rank of~\( p \) is defined as
\[ 
    \rank_\tau(p)\defeq \tau(\kappa_0(p)).
\]
\item Let~\( L\in A \) be an invertible element such that~\(
\chi_{(0,+\infty)}(L) \) and~\( \chi_{(-\infty,0)}(L) \) are both~\( \tau
\)-finite. Then the~\( \tau \)-signature of~\( L \) is defined as
\[ 
    \sig_\tau(L)\defeq
    \tau(\chi_{(0,+\infty)}(L))-\tau(\chi_{(-\infty,0)}(L)).
\]
\end{enumerate}
\end{definition}

The following lemma generalises
\Cref{lem:signature_quasi-projection} where~\( A=\Bdd(\mathcal{K}) \) for
\( \mathcal{K} \) a finite-dimensional Hilbert space, and~\( \tau \) is the
standard trace on it..

\begin{lemma}
Let~\( A \) be a unital \Cst-algebra and~\(
\tau \) be a densely defined trace on~\( A \).
Let~\( p\in A \) be a quasi-projection such that both~\( \kappa_0(p) \) and
\( \kappa_0(1_A-p) \) are~\( \tau \)-finite. Then  
\[ 
    \sig_\tau(2p-1_A)=\rank_\tau(p)-\rank_\tau(1_A-p).
\]
\end{lemma}

Now we shall prove an analogue of \Cref{thm:spectral_localiser} in this
setting, replacing signature and rank 
by their~\( \widehat{\tau} \)-versions. 

\begin{theorem}\label{thm:spectral_localiser_semifinite}
For every~\( t,\lambda>0 \),~\( \sig_{\widehat{\tau}}(L^\tau_{t^{-1},\lambda}) \),~\(
\rank_{\widehat{\tau}}(p^e_{t,\lambda}) \) and~\( \rank_{\widehat{\tau}}(p^f_{t,\lambda}) \)
are well-defined and satisfy the equality
\[ 
    \frac{1}{2}\sig_{\widehat{\tau}}(L^{\tau}_{t^{-1},\lambda})=
    \rank_{\widehat{\tau}}(p^e_{t,\lambda})-
    \rank_{\widehat{\tau}}(p^f_{t,\lambda}).
\]
As a consequence, let~\( \varepsilon,\delta>0 \) satisfy~\(
\varepsilon+\delta<\frac{1}{400} \). For any choice~\( t,\lambda \)
satisfying~\(
t>4\varepsilon^{-1}\norm*{[D,v]} \) and~\( \lambda>t\delta^{-1}
\), the half--\( \widehat{\tau} \)-signature 
\( \frac{1}{2}\sig_{\widehat{\tau}}(L_{t^{-1},\lambda}) \) 
coincides with the~\( \tau \)-numerical index 
pairing~\( \tau_*(\braket{[v],[D]}) \).   
\end{theorem}

\begin{proof}
We have proven in \Cref{thm:semi-finite_index_pairing_subspace} that the
orthogonal projection~\( P_{\lambda}^\tau\colon
\widehat{\mathcal{H}}^{\tau}\to\mathcal{H}_{\lambda}^\tau \) is~\(
\widehat{\tau} \)-finite. Note that
\( L_{t^{-1},\lambda} \),~\( p^e_{t,\lambda} \) and
\( p^f_{t,\lambda} \) are all supported on~\( \mathcal{H}^\tau_{\lambda}
\). Therefore,
\[ 
    \kappa_0(p^e_{t,\lambda}),\quad\kappa_0(p^f_{t,\lambda}),\quad
    \chi_{(0,+\infty)}(L),\quad\chi_{(-\infty,0)}(L)
\]
are all subprojections of~\( P_{\lambda}^{\tau} \), hence 
\( \widehat{\tau} \)-finite. So 
\Cref{def:tau_rank_sig} applies to them for the
\Cst-algebra~\( A=\Mat_2(\Cpt_{\widehat{\tau}}) \) with densely defined
trace~\( \widehat{\tau} \). The rest of the proof
follows in the same way as \Cref{thm:spectral_localiser}.
\end{proof}

\begin{remark}
It is well-known that numerical index pairings can be stated in the
framework of \emph{semi-finite index theory}
(cf.~\cite{Carey-Gayral-Rennie-Sukochev:Index_theory_locally_compact}). In
this general setting, a numerical index pairing is expressed as a pairing
between a \emph{semi-finite spectral triple} with a K-theory class.
The proof of \Cref{thm:commutative_diagram_trace} actually show the
following:

Let~\( A \) and~\( B \) be unital \Cst-algebras,~\( \tau \) a finite trace
on~\( B \), and~\( (\mathcal{A},E,D) \) an odd unbounded Kasparov~\( A \)-\(
B \)--module. By \Cref{lem:resolvent_tau_compact}, the data
\[ 
    (\mathcal{A},\mathcal{H}^\tau,D^\tau)\defeq 
    (\mathcal{A},E\otimes_B L^2(E,\tau),D\otimes\id)
\]
is a semi-finite spectral triple relative to~\( (\mathcal{N},\tau) \)
in the sense of
\cite{Carey-Phillips-Rennie-Sukochev:Local_index_formula_1}*{Definition
2.1},
where~\( \mathcal{N} \) is the enveloping von Neumann algebra of~\(
\Cpt_B(E) \) in the representation~\( \varrho^\tau \)
(cf.~\Cref{thm:commutative_diagram_trace}). 
Similar results have been achieved in
\cite{Goffeng-Rennie-Usachev:KMS_states}*{Theorem 2.12}.

The operators~\( L^\tau_{\kappa} \) and~\( L^\tau_{\kappa,\lambda} \)
of \Cref{thm:spectral_localiser_semifinite} were introduced by
Schulz-Baldes and Stoiber as a variant of the
spectral localiser that works for semi-finite index pairings. Thus
\Cref{thm:spectral_localiser_semifinite} recovers a special case of the
main theorem of \cite{SBaldes-Stoiber:Semifinite_spectral_localizer}.

Finally, we note that the setting in
this section is, albeit specific, yet enough to
cover many interesting examples that naturally emerge from the index theory
of \emph{tight-binding models} of topological insulators as in 
\cites{Kellendonk-Richter-SBaldes:Edge_current_channels,Bourne-Kellendonk-Rennie:Bulk-edge_correspondence,Bourne-Prodan:Noncommutative_Chern_numbers,Bourne-Mesland:Topological_phases},
in which the \Cst-algebra~\( A \) is given by the reduced \Cst-algebra of
an \'etale groupoid~\( \mathcal{G}\rightrightarrows \Omega \), either from
a crossed product by~\( \Z^d \) or as a transversal of a crossed product by
\( \R^d \), hence unital; 
and~\( B \) is the commutative, unital \Cst-algebra~\(
\Cont(\Omega) \). A finite, faithful
trace on~\( B \) thus corresponds to a finite, regular Borel measure on~\(
\Omega\)  with full support. These conditions 
are usually legitimated by physical assumptions.
\end{remark}

\setcounter{section}{0}
\renewcommand\theHsection{P2.\thesection}
\renewcommand\thesection{\Alph{section}}

\section{The asymptotic morphism}
\label[appendix]{app:asymptotic_morphism}
Given an \emph{essential} odd unbounded Kasparov~\( A \)-\( B \)--module~\(
(\mathcal{A},E,D) \), where essential means~\( \overline{\mathcal{A}E}=E
\), then the family~\( (\Phi^D_t)_{t\in[1,\infty)} \)
\[
  \Phi_{t}^{D}\colon\Co(\mathbb{R})\otimes A\to \Cpt_B(E), \quad f\otimes
a\mapsto f(t^{-1}D)a,
\]
defines an asymptotic morphism, hence an odd
E-theory class~\( \llbracket\Phi^D_t\rrbracket\in\E_1(A,B) \).  It is
alluded to in the literature that the natural transformation~\(
\KK_{1}(A,B)\to\E_{1}(A,B) \) is implemented by the map~\(
(\mathcal{A},E,D)\mapsto \llbracket\Phi^{D}_{t}\rrbracket \). We provide a
rigorous proof of this folklore statement in \Cref{thm:Phi^D}.

For even unbounded Kasparov modules, a realisation of the natural
transformation~\( \KK_0(A,B)\to\E_0(A,B) \) has been described by Connes
and Higson in their unpublished note
\cite{Connes-Higson:Almost_homomorphisms}*{Section 8}.  Their construction
coincides with~\( (\Phi_t^D)_{t\in [1,\infty)} \) up to the Bott
periodicity isomorphism~\( \KK_0(A,B)\simeq\KK_1(A,\Co(\R)\otimes B) \).
For~\( D \) a Dirac-type operator on a Riemannian manifold and~\( B=\C \),
asymptotic morphisms of the form~\( f\otimes a\mapsto f(t^{-1}D)a \) 
have also been studied by
Guentner \cite{Guentner:PhD_thesis}*{Section 4.4} and
Dumitra\c{s}cu \cite{Dumitrascu:K-homology}*{Section 3}. For the Dirac
element in KK-theory for proper group actions on the Hilbert space, 
such an asymptotic morphism
has been considered by Higson and Kasparov
\cite{Higson-Kasparov:E-theory_KK-theory_groups}*{Section 8}.

Another such realisation of the natural transformation~\(
\KK\Rightarrow\E \) has been given by Dumitra\c{s}cu
\cite{Dumitrascu:KE-theory} based on
a variant of E-theory for~\( \Z/2 \)-graded \Cst-algebras.  
In particular, we note that the family of operators~\( (F_t\defeq
t^{-1}D(1+t^{-2}D^2)^{-\sfrac{1}{2}})_{t\in [1,\infty)} \) in
\eqref{eq:wt_Ft_Pt} below gives an \emph{asymptotic Kasparov module}
in the sense of \cite{Dumitrascu:KE-theory}*{Definition 3.6}.

The proof presented here is motivated by
\cite{Higson-Kasparov:E-theory_KK-theory_groups}*{Section 8}.
Recall that the odd unbounded Kasparov module~\( (\mathcal{A},E,D) \) 
gives a semi-split
extension \eqref{eq:Kasparov_KK_to_extension} of \Cst-algebras.
Connes and Higson has introduced in \cite{Connes-Higson:E-theory}, 
a ``standard'' asymptotic morphism
associated to a semi-split extension of \Cst-algebras. 
Its induced map in K-theory coincides the one induced by the Kasparov
product \eqref{eq:KK_index_pairing}:

\begin{lemma}[\cite{Connes-Higson:E-theory}*{Lemme
10}]\label{lem:Connes-Higson_extension_to_E}
Let
\[ \begin{tikzcd}
I \arrow[r,tail] & A \arrow[r, two heads, "p"] & B \arrow[l, bend left=30,
"s"]
\end{tikzcd} \]
be an extension of \Cst-algebras, which is semi-split by a section~\( s\colon
B\to A \). Then:
\begin{enumerate}
\item 
Let~\( (u_t)_{t\in [1,\infty)} \) be a
approximate unit of~\( I \) which is quasi-central for~\( A \). Then
there is an
asymptotic morphism~\( (\Phi^{\mathrm{CH}}_t\colon\Co(0,1)\otimes B\to
I)_{t\in[1,\infty)} \)
given by
\[
\Phi^\mathrm{CH}_t\colon \Co(0,1)\otimes B\to I,\qquad f\otimes b\mapsto
f(u_t)s(b).
\]
\item If~\( (u_t)_{t\in[1,\infty)} \) and~\( (v_t)_{t\in [1,\infty)} \) are
two different approximate units of~\( I \) that are quasi-central for~\( A
\). Then their induced asymptotic morphisms are asymptotically equivalent.  
\item If the extension is induced by an odd unbounded Kasparov module as in
\eqref{eq:Kasparov_KK_to_extension}, then asymptotic morphism becomes
\( \Phi_t^\mathrm{CH}\colon \Co(0,1)\otimes A\to\Cpt_B(E) \) 
and its induced map~\( \K_0(\Co(0,1)\otimes A)\to\K_0(\Cpt_B(E)) \) 
coincides
with the \K-theoretic index pairing~\( [v]\mapsto\braket{[v],[D]} \) up
to an isomorphism~\( \K_0(\Co(0,1)\otimes A)\simeq\K_1(A) \). 
\end{enumerate}
\end{lemma}

Our first step is to simplify the Connes-Higson asymptotic morphism in the
non-unital case. In what follows, we define
\begin{equation}\label{eq:wt_Ft_Pt}
\begin{aligned}
    w_t&\defeq (1+t^{-2}D^2)^{-1},\\
    F_t&\defeq t^{-1}D(1+t^{-2}D^2)^{-\sfrac{1}{2}},\\
    P_t&\defeq\frac{1+F_t}{2}
    =\frac{1+t^{-1}D(1+t^{-2}D^2)^{-\sfrac{1}{2}}}{2}.
\end{aligned}
\end{equation}
We note that~\( F_1=D(1+D^2)^{-\sfrac{1}{2}} \) and~\(
P_1=\frac{1+D(1+D^2)^{-\sfrac{1}{2}}}{2} \).
This amounts to choosing the chopping function~\(
\chi(x)\defeq x(1+x^2)^{-\sfrac{1}{2}} \) in \eqref{eq:P_1}.

\begin{lemma}\label{lem:reduction_of_au}
Let~\( (\mathcal{A},E,D) \) be an odd unbounded 
Kasparov~\( A \)-\( B \)--module,
\( (v_{t})_{t\in[1,\infty)}\subseteq\mathcal{A} \) 
an approximate unit for~\( A \) that is quasicentral for~\( P_{1} \).
Then the representative~\( f\otimes a\mapsto
f(w_{t}^{1/2}v_{t}w_{t}^{1/2})P_{1}a \) of the Connes--Higson asymptotic
morphism is asymptotically equivalent to~\( f\otimes a\mapsto
f(w_{t})P_{1}a \).
\end{lemma}

\begin{proof}
Since~\( \overline{AE}=E \), it follows that~\( v_{t}\to 1 \) \st-strictly and
\( w_{t}^{1/2}v_{t}w_{t}^{1/2} \) is an approximate unit for~\( \Cpt_B(E) \),
quasi-central for~\( A \) and~\( P_{1} \).

Observe that~\( w_{t}P_1a=P_1w_{t}a\in\Cpt_B(E) \) and
\( w_{t}^{1/2}v_{t}w_{t}^{1/2}P_1a=w_{t}^{1/2}v_{t}P_1w_{t}^{1/2}a\in\Cpt_B(E) \),
so both asymptotic families extend to families~\( \Co(0,1]\otimes A\to
\Cpt_B(E) \) (but do not give asymptotic morphisms!). The~\( \Cst \)-algebra
\( \Co(0,1] \) is singlely generated by the function~\( f(x)=x \), 
so it suffices to check that 
\[
    \lim_{t\to\infty}\Phi_{t}(h\otimes a)-\Psi_{t}(h\otimes a)=0,
\] 
for the function~\( h(x)=x \). Now
\begin{align*}
\Phi_{t}(h\otimes a)-\Psi_{t}(h\otimes a)&=(w_{t}-w_{t}^{1/2}v_{t}w_{t}^{1/2})P_1a\\
&\sim (w_{t}-w_{t}^{1/2}v_{t}w_{t}^{1/2})aP_1\\
&\sim w_{t}^{1/2}((1-v_{t})a)w_{t}^{1/2}P_1\sim 0,
\end{align*}
which establishes the result.
\end{proof}

Our next goal is to establish that~\( f\otimes a\mapsto f(w_{t})a \) is
asymptotically homotopic to~\( f(P_{t})a \).

\begin{lemma}\label{lem:difference_construction} 
Let~\( A,B,C \) be~\( \Cst \)-algebras with~\( A,B\subseteq C \). Suppose that~\(
\{\mathfrak{p}_t\}_{t\in[1,\infty)}\) and~\( \{\mathfrak{q}_t\}_{t\in
[1,\infty)} \) are continuous families of contractions in~\( C \), such
that for all~\( f\in \Co(0,1) \) and~\( a\in A \):
\begin{center}
\(\left(f(\mathfrak{p}_t)-
f(\mathfrak{q}_t)\right)a\in B\)\quad and\quad 
\( [\mathfrak{p}_t-\mathfrak{q}_t,a]\to 0\) in
norm.
\end{center}
Then~\( \Psi^{\mathfrak{p},\mathfrak{q}}_{t}\colon f\otimes a\mapsto
(f(\mathfrak{p}_t)-f(\mathfrak{q}_t))a \) is an asymptotic morphism
\( \Co(0,1)\otimes A\to B \).  
\end{lemma}

\begin{proof}
Since for each~\( t\in[1,\infty) \) the map
\( \Psi^{\mathfrak{p},\mathfrak{q}}_{t} \) is a difference of linear
contractions, we have 
\( \norm*{\Psi_{t}^{\mathfrak{p},\mathfrak{q}}}\leq 2 \) for
all~\( t \), and thus it suffices to check that
\( \Psi_{t}^{\mathfrak{p},\mathfrak{q}} \) is an asymptotic morphism on a dense
subalgebra of~\( \Co(0,1)\otimes A \). The maps
\( \Psi_{t}^{\mathfrak{p},\mathfrak{q}} \) extend to maps 
\( \Co(0,1]\otimes A\to C \) and 
the function~\( g(x)=x \) generates a dense \st-subalgebra of~\( \Co(0,1] \)
by the Stone--Weierstrass theorem. For~\( g \) we can estimate
\begin{gather*}
\Psi^{\mathfrak{p},\mathfrak{q}}_{t}(g\otimes
a)\Psi^{\mathfrak{p},\mathfrak{q}}_{t}(g\otimes
b)=(\mathfrak{p}_t-\mathfrak{q}_t)a(\mathfrak{p}_t-\mathfrak{q}_t)b\sim
(\mathfrak{p}_t-\mathfrak{q}_t)^{2}ab=\Psi^{\mathfrak{p},\mathfrak{q}}_{T}(g^{2}\otimes
ab),\\
(\Psi^{\mathfrak{p},\mathfrak{q}}_{t}(g\otimes a))^{*}=a^{*}(\mathfrak{p}_t-\mathfrak{q}_t)\sim (\mathfrak{p}_t-\mathfrak{q}_t)a^{*}=\Psi^{\mathfrak{p},\mathfrak{q}}_{t}(g\otimes a^{*}),
\end{gather*}
and we conclude that~\( \Psi^{\mathfrak{p},\mathfrak{q}}_{t} \) is an
asymptotic morphism on the dense \st-subalgebra generated by~\( g \), and thus
an asymptotic morphism~\( \Psi^{\mathfrak{p},\mathfrak{q}}_{t}\colon
\Co(0,1]\otimes
A\to C \). Since~\( f(\mathfrak{p}_t)a, f(\mathfrak{q}_t)a\in B \) 
for~\( f\in \Co(0,1) \), it
restricts to an asymptotic morphism
\( \Psi^{\mathfrak{p},\mathfrak{q}}_{t}\colon \Co(0,1)\otimes A\to B \).
\end{proof}

\begin{lemma}
\label{lem:error_reduction}
Let~\( f\in \Co(0,1) \) and~\( a\in\End^{*}_B(E) \). Let~\( P\in\End^{*}_B(E) \)
satisfy~\( 0\leq P\leq 1 \) and
\( (P^{2}-P)a\in\Cpt_B(E) \).
Let~\( \{u_t\}_{t\in[1,\infty)} \) be a family of positive contractions in
\( \End^*_B(E) \) such that~\( u_t\to \id_E \) \st-strictly,
and~\( u_{t}P=Pu_{t} \) for all~\( t\in [1,\infty) \). Then
\begin{equation}\label{eq:convergence_error_reduction}
    f(u_{t}P)a-f(P)a-f(u_{t})Pa\to 0.
\end{equation}
\end{lemma}
\begin{proof} 
The functions~\( f_{n}(x)=x^{n}(1-x)=x^{n}-x^{n+1},\,\, n\geq 1 \) generate
\( \Co(0,1) \) (in fact,~\( n=1,2 \) suffices). Now
\begin{align*}
(f_{n}(u_{t}P)-f_{n}(P)-f_{n}(u_{t})P)a&=(u^{n}_{t}P^{n}-u^{n+1}_{t}P^{n+1}-P^{n}+P^{n+1}-(u_{t}^{n}-u^{n+1}_{t})P)a\\
&=(u^{n}_{t}-1)(P^{n}-P^{n+1})a+u^{n}_{t}(1-u_{t})(P^{n+1}-P)a,
\end{align*}
and since 
\begin{gather*}
(P^{n}-P^{n+1})a=P^{n-1}(P-P^{2})a\in\Cpt_B(E),\\
(P^{n+1}-P)a=\sum_{k=0}^{n-1}P^{k}(P^{2}-P)a\in\Cpt_B(E).
\end{gather*}
Since~\( u_t\to\id_E \) \st-strictly, it follows that~\( u_t^nT\to T \) in
norm for every~\( T\in\Cpt_B(E) \). Thus we have found the convergence in
\eqref{eq:convergence_error_reduction}.
\end{proof}

\begin{proposition}\label{prop:homotopy_asymptotic_morphism}
The asymptotic morphism
\begin{equation}\label{eq:Higson-Kasparov}\tag{HK}
    \Phi^D_t\colon\Co(0,1)\otimes A\to \Cpt_B(E),\quad
    f\otimes a\mapsto f(P_t)a
\end{equation}
is asymptotically homotopic to the Connes--Higson asymptotic morphism
\begin{equation}\label{eq:Connes-Higson}\tag{CH}
    \Phi_t^{\mathrm{CH}}\colon \Co(0,1)\otimes A\to \Cpt_B(E),\quad
    f\otimes a\mapsto f(u_t)P_1a,
\end{equation}
hence they induce the same \K-theoretic index pairing 
\( \K_1(A)\to\K_0(\Cpt_B(E)) \). 
\end{proposition}

\begin{proof}
By \Cref{lem:reduction_of_au}, the Connes--Higson asymptotic morphism
\eqref{eq:Connes-Higson} is
equivalent to~\( f\otimes a\mapsto f(w_{t})P_{1}a \). 
For each~\( t\in [1,\infty) \) we define the following families of maps 
\( \Co(0,1)\otimes A\to\Cpt_B(E) \):
\begin{align}
f\otimes a&\longmapsto f(w_t)P_1a, \label{eq:asymptotic_1}\tag{a1} \\
f\otimes a&\longmapsto f(w_tP_1)a-f(P_1)a,
\label{eq:asymptotic_2}\tag{a2}\\
f\otimes a&\longmapsto f(P_t)a, \label{eq:asymptotic_3}\tag{a3}
\end{align}
Then all of them give asymptotic morphisms. For \eqref{eq:asymptotic_1} and
\eqref{eq:asymptotic_3} this is well known. For \eqref{eq:asymptotic_2},
this follows from \Cref{lem:difference_construction}. It follows from
\Cref{lem:error_reduction} that \eqref{eq:asymptotic_1} is asymptotically
equivalent \eqref{eq:asymptotic_2}.

The proof will be complete if we show that \eqref{eq:asymptotic_2} is
homotopic to \eqref{eq:asymptotic_3}. To this end, 
we consider the map~\( [1,\infty)\times
[0,1]\to \End^{*}_B(E) \)
\[(t,r)\mapsto (rw_{t}+1-r)P_{(1-r)t+r}, \quad (t,r)\mapsto rP_{(1-r)t+r},\]
which is norm-continuous. Thus, we have a family of maps
\[H_{t}\colon \Co(0,1)\otimes A\to\Cont([0,1],\Cpt_B(E)),\]
given by
\[H_{t}(f\otimes a)(r)\defeq
f((rw_{t}+1-r)P_{(1-r)t+r})-f(rP_{(1-r)t+r}))a,\]
such that for each~\( f\otimes a \), the map~\( t\mapsto H_{t}(f\otimes a) \) is
continuous. To prove that~\( H_{t} \) is an asymptotic morphism, we will use
\Cref{lem:difference_construction}, so we have to show that
\[(rw_{t}+1-r)P_{(1-r)t+r}-rP_{(1-r)t+r}=(rw_{t}+1-2r)P_{(1-r)t+r},\]
is~\( A \)-quasi-central for the sup-norm on~\( \Cont([0,1],\Cpt_B(E)) \), that is,
\begin{equation}
\label{eq:sup_norm_limit}
\lim_{t\to\infty}\sup_{r\in[0,1]}\norm*{\left[(rw_{t}+1-2r)P_{(1-r)t+r}\,,\,a\right]}=0,\quad
\text{for every~\( a\in A \).}
\end{equation}
To this end, for each~\( a\in A \) and every~\( r\in [0,1] \) we have
\begin{multline*}
[(rw_{t}+1-2r)P_{(1-r)t+r}\,,\,a]\\
=r[w_{t},a]P_{(1-r)t+r}+(1-r)[P_{(1-r)t+r}\,,\,a]+r(w_{t}-1)[P_{(1-r)t+r}\,,\,a].
\end{multline*}
For the first summand, applying
\Cref{prop:commutator_sqrt_ctst} with~\( s=1 \) gives
\[\norm*{r[w_{t},a]P_{(1-r)t+r}}\leq 2t^{-1}\norm{[D,a]}.\]
For the second summand, we observe that~\( [P_{t},a]=\frac{1}{2}[F_{t},a] \).
Setting~\( s=0 \) in \Cref{prop:commutator_Ds_Ft} gives  
\begin{align*}
\norm*{(1-r)[P_{(1-r)t+r},a]}\leq
\frac{3}{2}\norm{[D,a]}(1-r)((1-r)t+r)^{-1}\leq
\frac{3}{2}t^{-1}\norm{[D,a]}.
\end{align*}
For the last term, observe that~\( 1-w_{t}=t^{-2}D^{2}(1+t^{-2}D^{2})^{-1} \)
and for any~\( 0<s<1 \), \Cref{prop:commutator_Ds_Ft} now yields
\begin{align*}
\norm*{r(1-w_{t})[P_{(1-r)t+r},a]}&=rt^{-s}\norm*{t^{-2+s}\abs{D}^{2-s}(1+t^{-2}D^{2})^{-1}\abs{D}^s[P_{(1-r)t+r},a]}\\
&\leq t^{-s}C_s\norm{[D,a]}((1-r)t+r)^{s-1}\leq t^{-s}C_s\norm{[D,a]},
\end{align*}
where
\[ 
    C_s\defeq
    1+\frac{2\Gamma(\tfrac{1-s}{2})}{\sqrt{\pi}\Gamma(\tfrac{2-s}{2})}.
\]
Thus, the convergence in \eqref{eq:sup_norm_limit} is established and~\( H_{t} \) gives the desired homotopy.     
\end{proof}

Thus we have proven the following result:

\begin{theorem}\label{thm:Phi^D}
Let~\( (\mathcal{A},E,D) \) be an essential odd 
unbounded Kasparov~\( A \)-\( B \)--module.
Then the asymtotic morphism
\[ \Phi^{D}_{t}\colon\Co(\mathbb{R})\otimes A\to \Cpt_B(E),\quad 
f\otimes a\mapsto f(t^{-1}D)a \]
represents the class of the image of~\( [D]\in\KK_1(A,B) \) in~\( \E_1(A,B)
\) under the natural transformation~\( \KK_1\to\E_1 \). In particular,
\( \Phi^D_t \) induces the same map 
\( \K_0(\mathrm{S}A)\to\K_0(\Cpt_B(E)) \) as~\(
[D]\in\KK_1(A,B) \), up to an isomorphism 
\( \K_1(A)\simeq\K_0(\mathrm{S}A) \).   
\end{theorem}

\begin{proof}
Note that
\[
    \chi\colon \R\to(0,1),\quad \chi(x)\defeq x(1+x^2)^{-\sfrac{1}{2}}
\]
is monotonically increasing and continuous. Thus it 
induces an isomorphism of \Cst-algebras 
\( \Co(0,1)\simeq\Co(\R) \). The
claim follows from the fact that \eqref{eq:Higson-Kasparov} is the
pullback of \eqref{eq:asymptotic_3} along~\( \chi \). 
\end{proof}

\section{Some commutator estimates}
\label[appendix]{app:commutator_estimates}
Let~\( (\mathcal{A},E,D) \) be an odd unbounded Kasparov
module and let~\( a\in\mathcal{A}\). 
Then~\( [D,a] \) extends to a bounded, adjointable operator on~\( E \).
We shall provide the following commutator estimates in
\Cref{prop:commutator_ct_st,prop:commutator_sqrt_ctst,prop:commutator_Ds_Ft}.

\begin{enumerate}
\item Let~\( c(x),s(x)\colon \R\to(0,1) \) be functions as in
\eqref{eq:choice_cs}:
\[
    c(x)\defeq\sqrt{\frac{1}{2}-\frac{1}{2}x(1+x^2)^{-\sfrac{1}{2}}},\quad 
    s(x)\defeq\sqrt{\frac{1}{2}+\frac{1}{2}x(1+x^2)^{-\sfrac{1}{2}}},
\]
and let~\( c_t\defeq c(t^{-1}D) \),~\( s_t\defeq c(t^{-1}D) \). Then the
following estimate holds:  
\[ 
    \norm*{[c_t,a]}\leq\frac{1}{2}t^{-1}\norm*{[D,a]},\quad
    \norm*{[s_t,a]}\leq\frac{1}{2}t^{-1}\norm*{[D,a]},\quad
    \norm*{[\sqrt{c_ts_t},a]}\leq
    2t^{-1}\norm*{[D,a]}.
\]
\item Let~\( F_t\defeq t^{-1}D(1+t^{-2}D^2)^{-\sfrac{1}{2}} \). Then for
all~\( 0\leq s<1 \),~\( [F_t,a] \) maps~\( E \) into~\( \dom\abs{D}^s \)
and the following estimate holds:
\[\norm*{\abs{D}^{s}[F_{t},a]}\leq C_s t^{s-1} \norm*{[D,a]},\]
where
\[ 
    C_s\defeq
    1+\frac{2\Gamma(\tfrac{1-s}{2})}{\sqrt{\pi}\Gamma(\tfrac{2-s}{2})}.
\]
Here~\( \Gamma(\alpha) \) is the Gamma function
\eqref{eq:Gamma_function}. 
\end{enumerate}

Then (1) completes the proof of \Cref{thm:choice_cs}, and 
(2) completes the proof of \Cref{prop:homotopy_asymptotic_morphism} in
\Cref{app:asymptotic_morphism}.

Our main tool is the following lemma:

\begin{lemma}[\citelist{\cite{G-Bondia_Varilly_Joseph_Figueroa:Elements_of_NCG}*{Lemma
10.15}; \cite{Bratelli-Robinson:Operator_algebras_quantum_statistics_1}*{Theorem
3.2.32}}]\label{lem:commutator_estimate}
Let~\( B \) be a \Cst-algebra and 
\( \mathcal{E} \) be a Hilbert~\( B \)-module.  
Let~\( a\in\End^*_B(\mathcal{E}) \) and~\( \mathcal{D} \)
be a self-adjoint, regular, possibly unbounded operator on~\( \mathcal{E} \)
such that~\( [\mathcal{D},a] \) extends to an element in~\( \End^*_B(\mathcal{E}) \).
Then:
\begin{enumerate}
\item \(\norm*{[\mathrm{e}^{\mathrm{i}\mathcal{D}},a]}\leq\norm*{[\mathcal{D},a]}\).
\item Let~\( f \) be a function on~\( \R \) with Fourier transform
\( \widehat{f} \), such that
\[
C_f\defeq
\frac{1}{2\pi}\int_{-\infty}^{\infty}\abs*{\xi\widehat{f}(\xi)}\,\mathrm{d}\xi
=\frac{1}{2\pi}\int_{-\infty}^{\infty}\abs*{\widehat{f'}(\xi)}\,\mathrm{d}\xi
<+\infty.
\]
Then~\( \norm*{[f(\mathcal{D}),a]} \) is bounded and satisfies the estimate
\[ 
    \norm*{[f(\mathcal{D}),a]}\leq C_f\cdot\norm*{[\mathcal{D},a]}.
\]
\end{enumerate}
\end{lemma}

To compute the commutators~\( [c_t,a] \) and~\( [s_t,a] \), we will use the
following trigonometric substitutions.

\begin{lemma}
Define the function
\( u\colon \R\to\left[0,\pi\right] \) by~\(u(x)\defeq 
\frac{\pi}{2}+\arctan(x)\).
Then~\( c(x)=\cos\left(\frac{1}{2}u(x)\right) \) and~\(
s(x)=\sin\left(\frac{1}{2}u(x)\right) \).  
\end{lemma}

\begin{proof}
Recall that for~\( \theta\in[0,\pi] \), we have the following half-angle
formulae of trigonometric functions:
\[ 
    \cos\left(\frac{\theta}{2}\right)=\sqrt{\frac{1+\cos(\theta)}{2}},\quad
    \sin\left(\frac{\theta}{2}\right)=\sqrt{\frac{1-\cos(\theta)}{2}},\quad
    \text{for~\( \theta\in[0,\pi] \).} 
\]
Also,
\[ 
    \cos\left(\frac{\pi}{2}+\arctan x\right)=-\sin(\arctan
    x)=-\frac{x}{\sqrt{1+x^2}}.
\]
Thus we have
\begin{align*}
\cos\left(\frac{1}{2}u(x)\right)&=\sqrt{\frac{1+\cos(\frac{\pi}{2}+\arctan(x))}{2}}
=\sqrt{\frac{1}{2}-\frac{1}{2}x(1+x^2)^{-\sfrac{1}{2}}},\\
\sin\left(\frac{1}{2}u(x)\right)&=\sqrt{\frac{1-\cos(\frac{\pi}{2}+\arctan(x))}{2}}
=\sqrt{\frac{1}{2}+\frac{1}{2}x(1+x^2)^{-\sfrac{1}{2}}}.\qedhere
\end{align*}
\end{proof}

Let~\( u_t\defeq u(t^{-1}D)=\frac{\pi}{2}+\arctan t^{-1}D \). Then the
preceding lemma says that~\( c_t=\cos(\frac{1}{2}u_t) \) and~\(
s_t=\cos(\frac{1}{2}u_t) \).   

\begin{lemma}\label{lem:commutator_ut}
We have the following estimate\textup{:}
\[ 
    \norm*{[u_t,a]}\leq t^{-1}\norm*{[D,a]}.
\]
\end{lemma}

\begin{proof}
Note that~\( u'(x)=\frac{1}{1+x^2} \) and
\[
    \widehat{\frac{1}{1+x^2}}(\xi)=\pi\mathrm{e}^{-\abs*{\xi}}.
\]
Thus
\[
C_u
\defeq\frac{1}{2\pi}\int_{-\infty}^{\infty}\abs*{\xi\widehat{u}(\xi)}\,\mathrm{d}\xi
=\frac{1}{2\pi}\int_{-\infty}^{\infty}\abs*{\widehat{u'}(\xi)}\,\mathrm{d}\xi
=\frac{1}{2\pi}\int_{-\infty}^{+\infty}\pi\mathrm{e}^{-\abs*{\xi}}\mathrm{d}\xi=\int_{0}^{+\infty}\mathrm{e}^{-\xi}\mathrm{d}\xi
=1.
\]
So we obtain
\[ 
    \norm*{[u_t,a]}\leq
    C_u\cdot\norm{[t^{-1}D,a]}=t^{-1}\norm*{[D,a]}.\qedhere
\]
\end{proof}

\begin{proposition}\label{prop:commutator_ct_st}
We have the following estimates\textup{:}
\[ 
    \norm{[c_t,a]}\leq \frac{1}{2}t^{-1}\norm*{[D,a]},\quad 
    \norm{[s_t,a]}\leq \frac{1}{2}t^{-1}\norm*{[D,a]}
\]
\end{proposition}

\begin{proof}
We have
\begin{align*} 
c_t&=\cos\left(\frac{1}{2}u_t\right)
=\frac{1}{2}\left(\mathrm{e}^{\frac{\mathrm{i}}{2}u_t}+
\mathrm{e}^{-\frac{\mathrm{i}}{2}u_t}\right), \\
s_t&=\sin\left(\frac{1}{2}u_t\right)
=\frac{1}{2\mathrm{i}}\left(\mathrm{e}^{\frac{\mathrm{i}}{2}u_t}-
\mathrm{e}^{-\frac{\mathrm{i}}{2}u_t}\right).
\end{align*}
Since the operator~\( \tfrac{1}{2}u_t \) is self-adjoint, we may apply
\Cref{lem:commutator_estimate} to it, which gives
\begin{multline*} 
\norm{[c_t,a]}\leq
\frac{1}{2}\norm{[\mathrm{e}^{\tfrac{\mathrm{i}}{2}u_t},a]}+
\frac{1}{2}\norm{[\mathrm{e}^{-\tfrac{\mathrm{i}}{2}u_t},a]}
\leq\frac{1}{4}\left(\norm*{\left[
u_t,a\right]}+\norm*{\left[u_t,a\right]}\right) \\
\leq \frac{1}{2}\norm*{[u_t,a]}\leq \frac{1}{2}t^{-1}\norm*{[D,a]}.
\end{multline*}
The same estimate holds for~\( [s_t,a] \) as well. 
\end{proof}

The following results on special functions, 
specifically, Beta functions and Gamma
functions, are needed for computing the
constant factors in the two estimates below. Recall that
(cf.~\cite{Andrews-Askey-Roy:Special_functions}*{Definition 1.1.3}) the
\emph{Beta function}~\( \operatorname{B}(\alpha,\beta) \) is defined as the
integral
\[ 
\operatorname{B}(\alpha,\beta)\defeq \int_{0}^1
t^{\alpha-1}(1-t)^{\beta-1}\,\mathrm{d}t,\quad\text{for~\(\alpha,\beta\in\C\) with~\(\Re
\alpha>0,\,\Re\beta>0 \)};
\]
and the \emph{Gamma function}~\( \Gamma(\alpha) \) is defined as
(cf.~\cite{Andrews-Askey-Roy:Special_functions}*{Corollary 1.1.5}):
\begin{equation}\label{eq:Gamma_function}
    \Gamma(\alpha)\defeq
    \int_{0}^{\infty}t^{\alpha-1}e^{-t}\,\mathrm{d}t,\quad\text{for~\(
    \alpha\in\C \) with~\( \Re \alpha>0 \).} 
\end{equation}

The Gamma function satisfies the following \emph{Euler's reflection
formula} (cf.~\cite{Andrews-Askey-Roy:Special_functions}*{Theorem 1.2.1}):
\begin{equation}\label{eq:Euler_reflection}
\Gamma(\alpha)\Gamma(1-\alpha)=\frac{\pi}{\sin\pi\alpha}. 
\end{equation}
In particular, setting~\( \alpha=\tfrac{1}{2} \) yields~\(
\Gamma(\tfrac{1}{2})=\sqrt{\pi} \). 

The relation between the Beta function and the Gamma function is
well-known to be (cf.~\cite{Andrews-Askey-Roy:Special_functions}*{Theorem
1.1.4}):
\[ 
    \operatorname{B}(\alpha,\beta)=\frac{\Gamma(\alpha)\Gamma(\beta)}{\Gamma(\alpha+\beta)}.
\]
We substitute~\( \lambda=\frac{t}{1-t} \) in the expression of the Beta
function to get
\begin{equation}\label{eq:Beta_function}
\begin{aligned}
\operatorname{B}(\alpha,\beta)&=\int_{0}^{\infty}\left(\frac{\lambda}{\lambda+1}\right)^{\alpha-1}\left(\frac{1}{\lambda+1}\right)^{\beta-1}\frac{1}{(\lambda+1)^2}\,\mathrm{d}\lambda\\
&=\int_{0}^{\infty}\lambda^{\alpha-1}(\lambda+1)^{-\alpha-\beta}\,\mathrm{d}\lambda.  
\end{aligned}
\end{equation}

\begin{proposition}\label{prop:commutator_sqrt_ctst}
For all~\( 0<s\leq 1 \), the following estimate holds\textup{:}
\[ 
    \norm{[(1+t^{-2}D^2)^{-s},a]}\leq 2t^{-1}\norm{[D,a]}.
\]
As a consequence, we have
\[ 
    [\sqrt{c_ts_t},a]\leq 2t^{-1}\norm*{[D,a]}.
\]
\end{proposition}

\begin{proof}
Write~\( D_t\defeq t^{-1}D \) for abbreviation. 
For~\( s=1 \), we have
\begin{align*} 
[(1+D_t^2)^{-1},a]=&-(1+D_t^2)^{-1}[D_t^2,a](1+D_t^2)^{-1} \\
=&-(1+D_t^2)^{-1}D_t[D_t,a](1+D_t^2)^{-1}\\
&-(1+D_t^2)^{-1}[D_t,a]D_t(1+D_t^2)^{-1}.
\end{align*}
Since both~\( (1+D_t^2)^{-1}D_t \) and~\( (1+D_t^2)^{-1} \) are contractive
by functional calculus of~\( D_t \), we have
\[ 
    \norm{[(1+D_t^2)^{-1},a]}\leq2\norm*{[D_t,a]}=2t^{-1}\norm*{[D,a]}.
\]

For~\( 0<s<1 \), recall the following integral formula
(cf.~\cite{Carey-Phillips:Unbounded_Fredholm_modules}*{Remark 3}):
\begin{equation}\label{eq:integral_formula}
    (1+D_t^2)^{-s}=\frac{\sin(s\pi)}{\pi}\int_{0}^{\infty}\lambda^{-s}(1+D_t^2+\lambda)^{-1}\mathrm{d}\lambda,\quad
    0<s<1.
\end{equation}
This gives
\begin{align*}
[(1+D_t^2)^{-s},a]=&\frac{\sin(s\pi)}{\pi}\int_{0}^{\infty}\lambda^{-s}[(1+\lambda+D_t^2)^{-1},a]\mathrm{d}\lambda
\\
=&-\frac{\sin(s\pi)}{\pi}\int_{0}^{\infty}\lambda^{-s}(1+\lambda+D_t^2)^{-1}[1+\lambda+D_t^2,a](1+\lambda+D_t^2)^{-1}\mathrm{d}\lambda \\
=&-\frac{\sin(s\pi)}{\pi}\int_{0}^{\infty}\lambda^{-s}(1+\lambda+D_t^2)^{-1}[D_t^2,a](1+\lambda+D_t^2)^{-1}\mathrm{d}\lambda \\
=&-\frac{\sin(s\pi)}{\pi}\int_{0}^{\infty}\lambda^{-s}(1+\lambda+D_t^2)^{-1}D_t[D_t,a](1+\lambda+D_t^2)^{-1}\mathrm{d}\lambda \\
&-\frac{\sin(s\pi)}{\pi}\int_{0}^{\infty}\lambda^{-s}(1+\lambda+D_t^2)^{-1}[D_t,a]D_t(1+\lambda+D_t^2)^{-1}\mathrm{d}\lambda.
\end{align*}
We note that~\( (1+\lambda+D_t^2)^{-1}D_t=D_t(1+\lambda+D_t^2)^{-1}
\) is contractive by functional calculus of~\( D_t \). Therefore,
\begin{align*}
\norm*{[(1+D_t^2)^{-s},a]}&\leq\frac{2\sin(s\pi)}{\pi}\norm{[D_t,a]}\int_{0}^{\infty}\lambda^{-s}(1+\lambda+D^2)^{-1}\mathrm{d}\lambda\\
&\leq\frac{2t^{-1}\sin(s\pi)}{\pi}\norm{[D,a]}\int_{0}^{\infty}\lambda^{-s}(1+\lambda)^{-1}\mathrm{d}\lambda,
\end{align*}
while we know from the expression \eqref{eq:Beta_function} of the Beta
function as well as the Euler reflection formula
\eqref{eq:Euler_reflection} that, for all~\( 0<s<1 \): 
\begin{align*} 
\int_{0}^{\infty}\lambda^{-s}(1+\lambda)^{-1}\mathrm{d}\lambda
=\operatorname{B}(1-s,s)=\frac{\Gamma(s)\Gamma(1-s)}{\Gamma(1)}
=\frac{\pi}{\sin s\pi}.
\end{align*}
Then we conclude that
\[ 
    \norm{[(1+D_t^2)^{-s},a]}\leq 2t^{-1}\norm*{[D,a]},\quad 0<r\leq1.
\]
Now the estimate for~\( \norm{[\sqrt{c_ts_t},a]} \) follows from~\(
\sqrt{c_ts_t}=\frac{1}{\sqrt{2}}(1+D_t^2)^{-\sfrac{1}{4}} \).
\end{proof}

The estimate below follows similarly.

\begin{proposition} 
\label{prop:commutator_Ds_Ft}
Let~\( F_{t}\defeq t^{-1}D(1+t^{-2}D^{2})^{-\sfrac{1}{2}} \). Then for 
all~\( 0\leq s<1 \),~\( [F_{t},a] \) maps~\( E \) into~\( \dom\abs{D}^{s} \), and 
the following estimate holds\textup{:}
\[\norm*{\abs{D}^{s}[F_{t},a]}\leq t^{s-1} C_s\norm*{[D,a]},\]
where
\[ 
    C_s\defeq
    1+\frac{2\Gamma(\tfrac{1-s}{2})}{\sqrt{\pi}\Gamma(\tfrac{2-s}{2})}.
\]
In particular, setting~\( s=0 \) gives  
\[ 
    \norm*{[F_t,a]}\leq 3t^{-1}\norm*{[D,a]}.
\]

\end{proposition}

\begin{proof}
Write~\( D_t\defeq t^{-1}D \) for abbrevation. Then we have
\[ \abs{D}^{s}[F_{t},a]=\abs{D}^{s}(1+D_t^{2})^{-\sfrac{1}{2}}[D_t,a]
+\abs{D}^s[(1+D_t^{2})^{-\sfrac{1}{2}},a]D_t. \]
We address the two summands separately. For the first one it follows
that
\begin{align*}
\norm*{\abs{D}^{s}(1+D_t^{2})^{-\sfrac{1}{2}}[D_t,a]}
&=\norm*{\abs{D}^{s}t^{-1}(1+t^{-2}D^{2})^{-\sfrac{1}{2}}[D,a]}\\
&=
t^{s-1}\norm*{t^{-s}\abs{D}^{s}(1+t^{-2}D^{2})^{-\sfrac{1}{2}}[D,a]} \\
&\leq t^{s-1}\norm*{[D,a]},
\end{align*}
where the last inequality holds because~\(
t^{-s}\abs{D}^s(1+t^{-2}D^2)^{-\sfrac{1}{2}}=\abs{D_t}^s(1+D_t^2)^{-\sfrac{1}{2}} \) is contractive for~\( s<1 \).  

For the second summand, we expand~\( (1+D_t^{2})^{-\sfrac{1}{2}} \) with
the integral formula \eqref{eq:integral_formula} to get
\begin{align*}
&\abs{D}^{s}[(1+D_t^2)^{-\sfrac{1}{2}},a]D_t \\
=-\,&\frac{\abs{D}^{s}}{\pi}\int_{0}^{\infty}\lambda^{-\sfrac{1}{2}}(1+\lambda+D_t^2)^{-1}(D_t[D_t,a]+[D_t,a]D_t)(1+\lambda+D_t^2)^{-1}D_t\,\mathrm{d}\lambda\\
=-\,&\frac{t^{-1}\abs{D}^{s}}{\pi}\int_{0}^{\infty}\lambda^{-\sfrac{1}{2}}(1+\lambda+D_t^2)^{-1}(D_t[D,a]+[D,a]D_t)(1+\lambda+D_t^2)^{-1}D_t\,\mathrm{d}\lambda.
\end{align*}
Since~\( D_t(1+\lambda+D_t^2)^{-1}D_t=(1+\lambda+D_t^2)^{-1}D_t^{2} \) is
contractive, it follows that
\begin{align*}
\norm*{\abs{D}^{s}[(1+D_t^2)^{-\sfrac{1}{2}},a]D_t}
&\leq\frac{2t^{-1}}{\pi}\cdot\norm{[D,a]}\cdot\norm*{\int_{0}^{\infty}\abs{D}^s\lambda^{-\sfrac{1}{2}}(1+\lambda+D_t^2)^{-1}\,\mathrm{d}\lambda}\\
&=\frac{2t^{s-1}}{\pi}\cdot\norm{[D,a]}\cdot\norm*{\int_{0}^{\infty}\lambda^{-\sfrac{1}{2}}|D_t|^{s}(1+\lambda+D_t^2)^{-\frac{s}{2}}(1+\lambda+D_t^2)^{\frac{s-2}{2}}\mathrm{d}\lambda}\\
&\leq\frac{2t^{s-1}}{\pi}\cdot\norm{[D,a]}\cdot\norm*{\int_{0}^{\infty}\lambda^{-\sfrac{1}{2}}(1+\lambda+D_t^2)^{\frac{s-2}{2}}\mathrm{d}\lambda}\\
&\leq
\frac{2t^{s-1}}{\pi}\cdot\norm{[D,a]}\cdot\abs*{\int_0^\infty\lambda^{-\sfrac{1}{2}}(1+\lambda)^{\frac{s-2}{2}}\mathrm{d}\lambda}.
\end{align*}
Now by \eqref{eq:Beta_function} we have
\[ 
    \int_{0}^\infty\lambda^{-\sfrac{1}{2}}(1+\lambda)^{\frac{s-2}{2}}\,\mathrm{d}\lambda=\operatorname{B}(\tfrac{1}{2},\tfrac{1-s}{2})=\frac{\Gamma(\tfrac{1}{2})\Gamma(\tfrac{1-s}{2})}{\Gamma(\tfrac{2-s}{2})},
\]
whence
\begin{align*}
\norm*{\abs{D}^s[F_t,a]}&\leq
t^{s-1}\norm{[D,a]}+\frac{2t^{s-1}}{\sqrt{\pi}}\norm{[D,a]}\operatorname{B}(\tfrac{1}{2},\tfrac{1-s}{2})
\\
&=t^{s-1}\norm{[D,a]}\cdot\left(1+\frac{2\Gamma(\tfrac{1-s}{2})}{\sqrt{\pi}\Gamma(\tfrac{2-s}{2})}\right).\qedhere
\end{align*}
\end{proof}

\section{Numerical index pairings as traces}
\label[appendix]{app:proof_commutativity}
In this appendix, we shall provide the proof of
\Cref{thm:commutative_diagram_trace}:

\begin{theorem*}
The \st-homomorphism~\( \varrho^{\tau}\colon
\End_B^*(E)\to\Bdd(\mathcal{H}^\tau) \)
maps~\( \Cpt_B(E) \) into~\( \Cpt_{\widehat{\tau}} \). Its induced map
\[ 
    \varrho^{\tau}_*\colon \K_0(\Cpt_B(E))\to\K_0(\Cpt_{\widehat{\tau}})
\]
makes the the following diagram commute:
\begin{equation}\label{eq:commutative_diagram_trace}\tag{TR}
\begin{tikzcd}[row sep=large, column sep=large]
\K_1(A) \arrow[r, "\times{[D]}"] \arrow[rd, "\Phi^D_*"'] & \K_0(B)
\arrow[r, "\tau_*"] & \R \\
& \K_0(\Cpt_B(E))  \arrow[u, "{[E]}"]  \arrow[r,"\varrho^{\tau}_*"] &
\K_0(\Cpt_{\widehat{\tau}}) \arrow[u, "\widehat{\tau}_*"'].
\end{tikzcd} \end{equation}
\end{theorem*}

We recall the setup.
Let~\( A \) and~\( B \) be unital \Cst-algebras and~\( \tau\colon B\to\C \)
be a \emph{finite}, \emph{faithful} trace, which means that:
\begin{center}
\( \tau(b^*b)<+\infty \) for all~\( b\in B \)\quad and\quad
\( \tau(b^*b)=0 \) iff~\( b=0 \).
\end{center}
Let~\( (\mathcal{A},E,D) \) be a unital odd unbounded Kasparov~\( A \)-\( B
\)--module. The finite faithful trace~\( \tau\colon B\to\C \) transfers to a
semi-finite trace on~\( \Cpt_B(E) \), and generates an isometric
representation
\[ 
    \varrho^\tau\colon \End_B^*(E)\to\Bdd(\mathcal{H}^\tau),\quad T\mapsto
    T\otimes \id,
\]
where~\( \mathcal{H}_\tau \) is the Hilbert space~\( E\otimes_BL^2(B,\tau)
\), obtained as the localisation of~\( E \) at the GNS-representation~\(
B\to\Bdd(L^2(B,\tau)) \). The isometric representation also gives an
unbounded, self-adjoint operator~\( D^\tau \) on~\( \mathcal{H}^\tau \) as
the closure of~\( D\otimes\id \). 

We shall first prove that~\( \varrho^\tau \) maps~\( \Cpt_B(E) \) into~\(
\Cpt_{\widehat{\tau}} \), where~\( \Cpt_{\widehat{\tau}} \) is the
\Cst-algebra generated by all~\( \widehat{\tau} \)-finite projection inside
the enveloping von Neumann algebra~\( \mathcal{N} \) of~\(
\varrho^\tau(\End_B^*(E)) \).

The following lemma is of independent interest.

\begin{lemma}\label{lem:compactly_supported_fin_rank}
Let~\( \mathcal{D} \) be an unbounded self-adjoint regular operator on a Hilbert
\( B \)-module~\( \mathcal{E} \) such that 
\( (1+\mathcal{D}^2)^{-1}\in\Cpt_B(\mathcal{E}) \). 
Then for every~\( f\in\Cc(\R) \),~\( f(\mathcal{D})\in\Fin_B^*(\mathcal{E}) \).
\end{lemma}

\begin{proof} 
Recall that
(cf.~\cite{Blackadar:Operator_algebras}*{II.5.2.4})
the Pedersen ideal~\( \Ped(A) \) of 
a \Cst-algebra~\( A \) is a hereditary, dense
two-sided \st-ideal of~\( A \), generated as an ideal by elements
of the form~\( f(a) \) for~\( a\in A_+ \) and~\( f\in\mathfrak{F}_+ \),
where
\[ \mathfrak{F}_+\defeq\set*{f\in\Co[0,\infty)\;\middle|\;\text{\( f \)
vanishes on a neighbourhood of~\( 0 \)}}.
\]

We claim that~\( \varphi(\mathcal{D})\in\Ped(\Cpt_B(\mathcal{E})) \) for every
positive, even function~\( \varphi\in\Cc(\R) \). For such functions, there
exists~\( R\geq 0 \) such that~\( \varphi(x)=\varphi(-x)\geq 0 \) for all
\( x\in[-R,R] \); and~\( \varphi(x)=0 \) for all~\( x>R \). Then
\[ 
    \varphi(x)=\big(\varphi\circ\theta\big)\big((1+x^2)^{-1}\big)
\]
where~\( \theta \) is the function 
\[ 
    \theta\colon(0,1]\to[0,+\infty),\quad \theta(y)\defeq \sqrt{y^{-1}-1}.
\]
Since~\( \varphi(x) \) vanishes for all~\( x>R \) , the composition~\(
\varphi\circ\theta(y) \) vanishes for all~\( y\in(0,\tfrac{1}{R^2+1})
\), thus~\( \varphi\circ\theta\in\mathfrak{F}_+ \).
By assumption,~\( (1+\mathcal{D}^2)^{-1}\in\Cpt_B(\mathcal{E})_+ \). So
\( \varphi(\mathcal{D})=(\varphi\circ\theta)((1+\mathcal{D}^2)^{-1}) \) belongs to~\(
\Ped(\Cpt_B(\mathcal{E})) \). 

Now for every~\( f\in\Cc(\R) \), write~\( f=\sum_{i=1}^nf_i \) as a finite
linear combination of positive functions~\( f_i\in\Cc(\R) \),
then~\( f_i(\mathcal{D})\in\Cpt_B(E)_+ \). For each~\( f_i \), define
\( \varphi_i(x)\defeq f_i(x)+f_i(-x) \). Then~\( \varphi_i \) is a
positive, even function in~\( \Cc(\R) \) that dominates~\(
f_i \), thus~\( 0\leq f_i(\mathcal{D})\leq\varphi_i(\mathcal{D})  \). Hence,~\(
f_i(\mathcal{D})\in\Ped(\Cpt_B(E)) \) by hereditarity of the Pedersen
ideal, and~\( f(\mathcal{D})=\sum_{i=1}^nf_i(\mathcal{D})\in\Ped(\Cpt_B(E)) \) as well.  

It follows from \cite{Ara:Morita_equivalence_Pedersen_ideals
}*{Theorem 1.6} that the Pedersen ideal of~\( \Cpt_B(\mathcal{E}) \) is
given by~\( \Fin_B^*(\mathcal{P}) \), where~\( \mathcal{P} \) is a
submodule of~\( \mathcal{E} \)
(cf.~\cite{Ara:Morita_equivalence_Pedersen_ideals}*{Definition 1.2}).
A finite-rank operator on a Hilbert submodule of~\( \mathcal{E} \)  
is also a finite-rank operator on~\( \mathcal{E} \), regardless of whether
or not this submodule is complemented. Then we conclude that for all~\(
f\in\Cc(\R) \), 
\[ 
f(\mathcal{D})\in\Ped(\Cpt_B(\mathcal{E}))=\Fin_B^*(\mathcal{P})\subseteq\Fin_B^*(\mathcal{E}).\qedhere
\]
\end{proof}

\begin{lemma}\label{lem:resolvent_tau_compact}
For every compact set~\( K\subseteq \R \),~\( \chi_K(D^\tau) \) is~\(
\widehat{\tau} \)-finite.    
As a consequence, 
one has~\( (D^\tau+\mathrm{i})^{-1}\in\Cpt_{\widehat{\tau}} \). 
\end{lemma}

\begin{proof}
Let~\( K\subseteq\R \) be a compact subset.
Choose any~\( \varphi\in\Cc(\R)_+ \) 
such that~\( \chi_K\leq\varphi \). Since
\( (\mathcal{A},E,D) \) is a unital odd, unbounded Kasparov module,~\(
(1+D^2)^{-1}\in\Cpt_B(E)_+ \). Thus by  
\Cref{lem:compactly_supported_fin_rank}, 
\( \varphi(D)\in\Fin_B^*(E) \). It follows from
(2) of \Cref{lem:transfer_trace} that
\( \varrho^\tau(\varphi(D))=\varphi(D^\tau) \) is~\( \widehat{\tau}
\)-finite. Therefore,
\[ 
    \widehat{\tau}(\chi_K(D^\tau))\leq\widehat{\tau}(\varphi(D^\tau))<+\infty,
\]
that is,~\( \chi_K(D^\tau) \) is a~\( \widehat{\tau} \)-finite projection,
hence belongs to~\( \Cpt_{\widehat{\tau}} \). In particular,~\(
\chi_{[-R,R]}(D^\tau)\in\Cpt_{\widehat{\tau}} \) for any~\( R>0 \).  

Since~\( \Cpt_{\widehat{\tau}} \) is an ideal in~\( \mathcal{N} \), it follows 
that
\[ 
    (D^\tau+\mathrm{i})^{-1}\chi_{[-R,R]}(D^\tau)\in\Cpt_{\widehat{\tau}}
\]
holds for every~\( R>0 \). Now let 
\( \{\mathrm{d}p_\lambda\}_{\lambda\in\R} \) be the spectral resolution of 
\( \mathcal{H}^\tau \) given by the self-adjoint operator~\( D^\tau \),
then
\[ 
    f(D^\tau)=\int_{\lambda\in\R}f(\lambda)\mathrm{d}p_\lambda
\]
holds for any Borel function~\( f\colon \R\to\C \). In particular, we have
\begin{align*} 
\norm*{(D^\tau+\mathrm{i})^{-1}-(D^\tau+\mathrm{i})^{-1}\chi_{[-\lambda,\lambda]}}&=\norm*{\int_{\lambda\in\R}(\lambda+\mathrm{i})^{-1}\,\mathrm{d}p_\lambda-\int_{\abs*{\lambda}\leq
R}(\lambda+\mathrm{i})^{-1}\,\mathrm{d}p_\lambda} \\
&=\norm*{\int_{\abs*{\lambda}>R}(\lambda+\mathrm{i})^{-1}\,\mathrm{d}p_{\lambda}}\\
&<\frac{1}{\sqrt{1+R^2}}.
\end{align*}
So~\( \{(D^\tau+\mathrm{i})^{-1}\chi_{[-R,R]}\}_{R>0} \) is a Cauchy net
that converges to~\( (D^\tau+\mathrm{i})^{-1} \) in norm. Hence,
\( (D^\tau+\mathrm{i})^{-1}\in\Cpt_{\widehat{\tau}} \) as well.  
\end{proof}

\begin{proposition}\label{prop:injective_K_tau}
The representation~\( \varrho^\tau\colon
\End_B^*(E)\to\Bdd(\mathcal{H}^\tau) \) maps~\( \Cpt_B(E) \) into~\(
\Cpt_{\widehat{\tau}} \).   
\end{proposition}

\begin{proof}
Recall that~\( (D+\mathrm{i})^{-1}\in\Cpt_B(E) \) generates~\( \Cpt_B(E)
\) as an ideal.
Therefore,~\( \varrho^\tau(\Cpt_B(E)) \) is contained in the norm closure
of the ideal generated by 
\[
    \varrho^\tau(\Cpt_B(E)(D+\mathrm{i})^{-1})=\varrho^\tau(\Cpt_B(E))(D^\tau+\mathrm{i})^{-1}.
\]
We have shown that~\( (D^\tau+\mathrm{i})^{-1}\in\Cpt_{\widehat{\tau}}
\). Since~\( \Cpt_{\widehat{\tau}} \) is a norm-cloded ideal of~\(
\mathcal{N} \), it therefore follows that~\(
\varrho^\tau(\Cpt_B(E)) \) is also contained in~\( \Cpt_{\widehat{\tau}}
\).  
\end{proof}

Having defined the injective \st-homomorphism, we are now able to prove the
commutativity of the diagram \eqref{eq:commutative_diagram_trace}.

\begin{theorem}
The diagram \eqref{eq:commutative_diagram_trace} commutes.
\end{theorem}

\begin{proof}
The subtlety in the proof comes from the fact that
projections in~\( \Mat_n(\Cpt_B(E)^+) \) may fail to be~\(
\widehat{\tau} \)-finite, whence the 
composition of maps 
\[ 
    \K_0(\Cpt_B(E))\xrightarrow{[E]}\K_0(B)\xrightarrow{\tau_*}\R
\]
is hard to compute.
To overcome this, we replace~\( \Cpt_B(E) \) by~\(
\Cpt\otimes B \) using Kasparov's stablisation theorem
\cite{Kasparov:K-functor_extension}.

Since~\( E \) is countably generated, there exists an isometry~\(
V\in\Hom^*_B(E,\mathcal{H}_B) \), where~\( \mathcal{H}_B\defeq
\ell^2(\N)\otimes B \) is the standard countably generated Hilbert~\( B
\)-module. It follows from functoriality of the interior tensor product (or
the KSGNS construction) that the following diagram commutes:
\begin{equation}\label{eq:commutative_diagram_bottom_right}
\begin{tikzcd}
\Cpt_B(\mathcal{H}_B) \arrow[r, "\varrho^\tau"] &
\Cpt_{\widehat{\tau}}(\mathcal{H}_B) \\
\Cpt_B(E) \arrow[u, "\Ad_V"] \arrow[r, "\varrho^\tau"] &
\Cpt_{\widehat{\tau}}(E) \arrow[u, "\Ad_V\otimes\id"'],
\end{tikzcd} 
\end{equation}
here we distinguish the norm-closed ideals~\( \Cpt_{\widehat{\tau}}(E) \)
and~\( \Cpt_{\widehat{\tau}}(\mathcal{H}_B) \), both generated 
by~\( \widehat{\tau} \)-finite projections, 
but on different Hilbert spaces~\(
E\otimes_B L^2(B,\tau)\) and~\( \mathcal{H}_B\otimes L^2(B,\tau) \),
whereas~\( \Cpt_{\widehat{\tau}}(E) \) can be identified via~\(
\Ad_V\otimes\id \) as a corner in~\( \Cpt_{\widehat{\tau}}(\mathcal{H}_B)
\).    

Now we claim that the following augmented diagram commutes:
\begin{equation}\label{eq:commutative_diagram_trace_augmented}\tag{TR+}
\begin{tikzcd}[row sep=large, column sep=large]
\K_1(A) \arrow[r, "\times{[D]}"] \arrow[rdd, bend right=40, "\Phi^D_*"'] & \K_0(B)
\arrow[r, "\tau_*"] & \R \\
& \K_0(\Cpt_B(\mathcal{H}_B)) \arrow[u, "{[\mathcal{H}_B]}"'] \arrow[r,
"\varrho^\tau_*"] & \K_0(\Cpt_{\widehat{\tau}}(\mathcal{H}_B))
\arrow[u, "\widehat{\tau}_*"']  \\
& \K_0(\Cpt_B(E)) \arrow[uu, bend left=60, "{[E]}"]
\arrow[u, "{(\Ad_V)_*}"']  \arrow[r,"\varrho^{\tau}_*"] &
\K_0(\Cpt_{\widehat{\tau}}) \arrow[u, "(\Ad_V)_*\otimes\id"'].
\end{tikzcd} \end{equation}

The commutativity of the left quadrant follows from following equality of
maps in K-theory:
\[ 
    [\mathcal{H}_B]\circ(\Ad_V)_*=[E].
\]
To see this, we view~\( \Ad_V(\Cpt_B(E))\subseteq\Cpt_B(\mathcal{H}_B) \) 
as a \Cst--\( \Cpt_B(E) \)-\( \Cpt\otimes B \)--correspondence, and~\(
\mathcal{H}_B \) as a imprimivity bimodule between~\(
\Cpt_B(\mathcal{H}_B) \) and~\( B \). Thus the equality follows from the
isomorphism of Hilbert \Cst-modules
\[ 
\Ad_V\Cpt_B(E)\otimes_{\Cpt_B(\mathcal{H}_B)}\mathcal{H}_B\simeq
\Ad_V\Cpt_B(E)\otimes_{\Ad_V\Cpt_B(\mathcal{H}_B)}V^*\mathcal{H}_B\simeq 
V^*\mathcal{H}_B=E.
\]

The bottom right square 
of \eqref{eq:commutative_diagram_trace_augmented} follows
from the commutativity of \eqref{eq:commutative_diagram_bottom_right} as
well as functoriality of K-theory. Thus, we are left to show that the top
right square commutes. Note that
\[
    \Cpt_B(\mathcal{H}_B)\simeq
    \Cpt\otimes B=\lim_{\longrightarrow}\Mat_n(B),
\]
and K-theory preserves direct limits. So every class in~\(
\K_0(\Cpt\otimes B) \) is represented by a formal differences of
projections~\( [p]-[q] \) where~\( p,q\in\Mat_{\infty}(B) \). 
Here we have used the assumption that
\( B \) is unital. It also follows that the induced K-theory map~\(
[\mathcal{H}_B] \) maps such a class to~\( [p]-[q] \),
viewed as an element in~\( \K_0(B) \).

The transferred trace~\( \widehat{\tau} \) on~\( \Fin^*_B(\mathcal{H}_B) \)
is, by \Cref{lem:transfer_trace}, given by
\[ 
    \widehat{\tau}(\ket{\varphi\otimes a}\bra{\psi\otimes b})
    \defeq
    \tau(\braket{\psi\otimes b,\varphi\otimes
    a})=\braket{\psi,\varphi}\cdot\tau(b^*a)
\]
for~\( a,b\in B \) and~\( \varphi,\psi\in\ell^2(\N) \).

Now let~\( (b_{ij})_{i,j=1}^n \) be a projection in~\( \Mat_n(B) \). Under
the isomorphism~\( \Cpt\otimes B\simeq\Cpt_B(\mathcal{H}_B) \), it gets
mapped to
\[
    \sum_{i,j=1}^n\ket{\delta_j\otimes b_{ij}}\bra{\delta_i\otimes 1}\in\Fin^*_B(\mathcal{H}_B),
\]
where~\( \delta_i\in\ell^2(\N) \) is the function~\( \N\to\C \) sending~\(
i\) to~\( 1 \) and~\( 0 \) otherwise. Thus we have
\begin{multline*} 
\widehat{\tau}\left(\sum_{i,j=1}^n\ket{\delta_j\otimes
b_{ij}}\bra{\delta_i\otimes 1}\right)=
\tau\left(\sum_{i,j=1}^n\braket{\delta_i\otimes 1,\delta_j\otimes
b_{ij}}\right)=\\
\sum_{i,j=1}^n\braket{\delta_i,\delta_j}\tau(b_{ij})=\sum_{i=1}^n\tau(b_{ii})=\tau((b_{ij})_{i,j=1}^n),
\end{multline*}
that is, the top right square of
\eqref{eq:commutative_diagram_trace_augmented} commutes, hence
the whole diagram commutes.
Now the commutativity of \eqref{eq:commutative_diagram_trace} 
follows from it.
\end{proof}

\bibliography{references}
\end{document}